\documentclass[10pt]{article}

\usepackage{amssymb,amsthm,amsmath,hyperref}

\numberwithin{equation}{section}

  \usepackage{paralist}
  \usepackage{graphics} 
  \usepackage{epsfig} 
\usepackage{graphicx}  \usepackage{epstopdf}

\newtheorem{theorem}{Theorem}[section]

\newtheorem{lemma}[theorem]{Lemma}
\newtheorem{proposition}{Proposition}[section]

\theoremstyle{definition}
\newtheorem{definition}[theorem]{Definition}
\newtheorem{remark}{Remark}[section]

\newcommand{\ep}{\varepsilon}

\newcommand\R{{\mathbb R}}
\newcommand{\fc}{\frac}
\numberwithin{equation}{section}

\allowdisplaybreaks

\begin{document}
\title{Global solutions for  $H^s$-critical nonlinear biharmonic \\ Schr\"{o}dinger equation}

\author{Xuan Liu, Ting Zhang\\
	\small{School of Mathematical Sciences, Zhejiang University,
		Hangzhou 310027, China}}
\date{}

\maketitle


\begin{abstract}
	We consider  the nonlinear biharmonic  Schr\"odinger equation $$i\partial_tu+(\Delta^2+\mu\Delta)u+f(u)=0,\qquad (\text{BNLS})$$
	in the critical Sobolev space $H^s(\R^N)$, where $N\ge1$, $\mu=0$ or $-1$, $0<s<\min\{\fc N2,8\}$ and $f(u)$ is a nonlinear function that behaves like $\lambda\left|u\right|^{\alpha}u$ with $\lambda\in\mathbb{C},\alpha=\frac{8}{N-2s}$. We prove the existence and uniqueness of the global solutions to (BNLS)  for the small initial data.
\end{abstract}
\section{Introduction}
In this paper, we consider the following  nonlinear biharmonic  Schr\"odinger equation
\begin{equation}\label{NLS}
	\begin{cases}i\partial_tu+(\Delta^2 +\mu\Delta) u+f(u)=0,\\
		u(0,x)=\phi(x),\end{cases}
\end{equation}
where  $t\in\mathbb{R}$, $x\in\mathbb{R}^N$, $N\ge1$,   $\phi\in H^s(\R^N)$, $0<s<\min \left\{\frac{N}{2},8\right\}$,   $\mu=-1$ or $\mu=0$, $u:\mathbb{R}\times\mathbb{R}^N\rightarrow\mathbb{C}$ is a complex-valued function and $f(u)$ is a nonlinear function that behaves like $\lambda \left|u\right|^{\alpha}u$ with $\lambda\in\mathbb{C}$, $\alpha>0$.
Note that if $\mu=0$ and $f(u)=\lambda \left|u\right|^{\alpha}u$ with $\lambda\in\mathbb{C},\alpha>0$,  the equation (\ref{NLS})  is invariant under the scaling, $
u_k(t,x)=k^{\frac{4}\alpha}u(k^4t,kx), k>0.
$  This means if $u$ is a solution of (\ref{NLS}) with the  initial datum $\phi$,  so is $u_k$ with the initial datum $\phi_k=k^{\frac{4}\alpha}\phi(kx)$.  Computing the homogeneous Sobolev norm, we get
$$
\left\|\phi_k\right\|_{\dot{H}^s}=k^{s-\frac{N}{2}+\frac{4}{\alpha}}\left\|\phi\right\|_{\dot{H}^s}.
$$
Hence the scale-invariant Sobolev space is $\dot{H}^{s_c}(\R^N)$,
with the critical index $s_c=\frac N2-\frac{4}\alpha$.  If $s_c=s$ (equivalently $\alpha=\frac{8}{N-2s}$),  the Cauchy problem (\ref{NLS}) is known as $H^s$-critical; if in particular $s_c=2$  (equivalently $\alpha=\frac{8}{N-4}$),  it is called energy-critical or $H^2$-critical.

The nonlinear  biharmonic  Schr\"odinger equation (\ref{NLS}), also called the fourth-order Schr\"odinger equation, was  introduced by Karpman \cite{Karpman1}, and Karpman--Shagalov \cite{Karpman2} to take into
account the role of small fourth-order dispersion terms in the propagation of intense laser beams in a bulk medium with Kerr
nonlinearity. The biharmonic Schr\"odinger equation has attracted a lot of interest in the past decade.  The  sharp dispersive estimates for the fourth-order Schr\"odinger operator in (\ref{NLS}), namely for the linear group associated to $i\partial_t+\Delta^2+\mu\Delta$ was obtained in  Ben-Artzi, Koch, and Saut \cite{Ben}. In \cite{Pa}, Pausader  established the corresponding Strichartz' estimate for the biharmonic  Schr\"odinger equation (\ref{NLS}).  Since then, the local and global well-posedness  for (\ref{NLS}) have been widely studied in recent years. See \cite{Guo3,Dinh,H,HLW1,HLW2,xuan,Miao2,Pa,Pa2,Wang} and references therein.

We are interested in the global solutions to (\ref{NLS}) in the critical Sobolev space $H^s\left(\R^N\right)$. For $s=2$, Pausader \cite{Pa} established  the global well-posedness  for the defocusing energy-critical equation (\ref{NLS})  (i.e. $\mu=0$ or $-1, f(u)=\lambda \left|u\right|^\alpha u$ with $\lambda>0,\alpha=\frac{8}{N-4}$)   in a radially symmetrical
setting. The global well-posedness  for the defocusing energy-critical problem without the radial condition and the  focusing energy-critical equation (\ref{NLS})  (i.e. $\mu=0$ or $-1$, $f(u)=\lambda \left|u\right|^\alpha u$ with $\lambda<0$, $\alpha=\frac{8}{N-4}$) were discussed in  \cite{Miao, Miao2, Pa3, Pa2}.  For general $s$, Y. Wang \cite{Wang} established the small radial initial data global solutions to the biharmonic Schr\"odinger equation by using the improved Strichartz's estimate for spatial spherically symmetric function. He  proved the global existence of solution for the Cauchy problem (\ref{NLS}) when $N\ge2$, $-\frac{3N-2}{2N+1}<s<\frac{N}{2},\alpha=\frac{8}{N-2s},\mu=0,f(u)=\lambda \left|u\right|^\alpha u,\lambda=\pm1$, and $\phi\in H^s(\R^N)$ is a small radial function.

The goal of this paper is to  establish the time global solution for (\ref{NLS}) with the small initial data in the critical Sobolev space $H^s(\R^N)$, where $N\ge1$, $0<s<\min \left\{\frac{N}{2},8\right\}$.
Before stating our results, we  define  the class $\mathcal{C}(\alpha)$.
\begin{definition}
	Let $\alpha>0$, $f\in C^{[\alpha]+1}(\mathbb{C},\mathbb{C})$ in the real sense,  where $[\alpha]$ denotes the largest integer less than or equal to $\alpha$, and  $f^{(j)}(0)=0$ for all $j$ with $0\leq j\leq [\alpha]$. We say that $f$ belongs to   the  class $\mathcal{C}(\alpha)$,  if it satisfies  one of the following two conditions:\\
	(i) $\alpha\notin \mathbb{Z}$, $f^{([\alpha]+1)}(0)=0$, and there exists $C>0$ such that for any $z_1,z_2\in\mathbb{C}$
	$$
	\left|f^{([\alpha]+1)}(z_1)-f^{([\alpha]+1)}(z_2)\right|\le
	C \left(\left|z_1\right|^{\alpha-[\alpha]}+\left|z_2\right|^{\alpha-[\alpha]}\right)\left|z_1-z_2\right|,
	$$
	(ii) $\alpha\in\mathbb{Z}$,  and there exists $C>0$ such that for any $z\in\mathbb{C}$
	$$
	\left|f^{([\alpha]+1)}(z)\right|\le C.
	$$
\end{definition}
\begin{remark}
	We note that the power type nonlinearity $f(u)=\lambda \left|u\right|^\alpha u$ or  $f(u)=\lambda \left|u\right|^{\alpha+1}$ with $\lambda\in\R, \alpha>0$ is of the class $\mathcal{C}(\alpha)$, which has been widely studied in classical and biharmonic nonlinear Schr\"odinger equations. See \cite{Ca9,Ca10,Guo3,Dinh,HLW1,HLW2,xuan,Miao,Miao2,Pa,Pa3,Pa2} for instance.
\end{remark}
\begin{remark}
	For any $\alpha>0$ and  $f\in \mathcal{C}(\alpha)$, it is easy to check that there exists $C>0$ such that for any $u,v\in\mathbb{C}$, we have
	\begin{equation}\label{fu}
		\left|f(u)-f(v)\right|\le C \left(\left|u\right|^{\alpha}+\left|v\right|^{\alpha}\right)\left|u-v\right|,\qquad \left|\partial_tf(u)\right|\le C \left|u\right|^{\alpha}\left|\partial_tu\right|.
	\end{equation}
\end{remark}
Our main result is the following. For the definitions of vector-valued Besov spaces $B^{\theta }_{q,2}\left(\R,L^r\left(\R^N\right)\right)$ and $B^{\theta-\sigma/4}_{q,2}B^\sigma_{r,2}$, we refer to Section \ref{s2}.
\begin{theorem}\label{T1}
	Assume  $0<s<\min\{8,\fc N2\}$, $N\ge1$, $\mu=0$ or $-1$, $f\in\mathcal{C}(\alpha)$ and $\alpha=\fc8{N-2s}>\alpha(s)$, where
	$$
	\alpha(s)=\begin{cases}
		0,\qquad \qquad\qquad \qquad \text{if } 0<s<4,\\
		\max \left\{\frac{s}{4}-1,s-5\right\}, \ \text{if } 4<s<8.
	\end{cases}
	$$
	Given any $\phi\in  H^s(\R^N)$ with $\left\|\phi\right\|_{H^s}$ sufficiently small, there exists  a unique  global solution
	$ u\in C\left(\R,H^s\left(\R^N\right)\right)\cap \mathcal{X}$
	to the  Cauchy problem (\ref{NLS}),  where
	$$
	\mathcal{X}=\left\{\begin{array}{ll}
		L^{q_1}\left(\R,B^{s}_{r_1,2}\right)\cap B^{s/4}_{q_1,2}\left(\R,L^{r_1}\left(\R^N\right)\right), &0<s\le4,\\
		L^{q_2}\left(\R,H^{4,r_2}(\R^N)\right),&s=4,\\
		L^{q_3}\left(\R,B^{s}_{r_3,2}\left(\R^N\right)\right)\cap B^{s/4}_{q_3,2}\left(\R,L^{r_3}\left(\R^N\right)\right)\cap H^{1,q_3}\left(\R,B^{s-4}_{r_3,2}\left(\R^N\right)\right),& 4<s<6,\\
		B^{s/4}_{2,2}\left(\R,L^{r_4}\left(\R^N\right)\right)\cap B^{\left(s-2\right)/4}_{2,2}\left(\R,L^{r_4}\left(\R^N\right)\right),& 6\le s<8,
	\end{array}\right.$$
	with
	$$
	\begin{cases}
		q_1=\frac{\left(2N+8\right)\left(N-2s+8\right)}{\left(N-2s\right)\left(N+8\right)}, \  \ &r_1=\frac{2N\left(N+4\right)\left(N-2s+8\right)}{8N\left(N+4\right)+\left(N-2s\right)\left(N^2-32\right)},\\
		q_2=\frac{2N-8}{N-8},\ &r_2=\frac{2N\left(N-4\right)}{N^2-8N+32},\\
		q_3=\frac{2\left(N-2s+8\right)}{N-2s},\ &r_3=\frac{2N\left(N-2s+8\right)}{\left(N-4\right)\left(N-2s\right)+8N},\\
		r_4=\frac{2N}{N-4}.
	\end{cases}
	$$
\end{theorem}
\begin{remark}
	Note that the lower bound $\alpha(s)$ is a continuous function of $s$. Moreover, the condition $\alpha>\max \left\{\frac{s}{4}-1,s-5\right\}$ is natural for $s>4$, since one time derivative corresponds to four spatial derivatives and the $s$-derivative of $u$ by the spatial variables requires the $(s-4)$-derivatives of $f(u)$ by (\ref{NLS}).
\end{remark}
\begin{remark}
	Theorem \ref{T1} improves the result in Y. Wang \cite{Wang} in the case $0<s<4$, where he made an additional radial assumption for the initial datum.
\end{remark}
Theorem \ref{T1}  may be considered as a generalization of the corresponding results for  the classical nonlinear Schr\"odinger equation. In \cite{Ca10}, Cazenave and Weissler showed the existence of the time global solutions for the small initial data of the $H^s$ critical Cauchy problem
\begin{equation}\label{CNLS}
	\begin{cases}i\partial_tu+\Delta u+\lambda|u|^\alpha u=0,\\
		u(0,x)=\phi(x)\in H^s(\mathbb{R}^N),\end{cases}
\end{equation}
for $0 \leq s<\fc N2$ and $[s]<\alpha=\fc4{N-2s}$. The condition $[s]<\alpha$ is the required regularity for $f(u)$, which can be improved to $s-1<\alpha$ by applying  the nonlinear estimates obtained in  Ginibre--Ozawa--Velo  \cite{G}, and Nakamura--Ozawa  \cite{Na3}.
Recently,   Nakamura--Wada \cite{Na,Na2} constructed some modified Strichartz's estimate  and  Strichartz type estimates in mixed Besov spaces   to obtain small global solutions with less regularity assumption for the nonlinear term. More precisely, they showed that if $1<s<4$,  $s\neq2$, $\alpha_0(s)<\alpha=\fc4{N-2s}$, with
$$
\alpha_{0}(s):=\left\{\begin{array}{ll}
	0, & \text { for } 0<s<2, \\
	\frac{s}{2}-1, & \text { for } 2<s<4,
\end{array}\right.
$$
the Cauchy problem (\ref{CNLS}) admits a unique time global solution for the small initial data. Theorem \ref{T1} extends the results in \cite{Na,Na2} into the biharmonic Schr\"odinger case.

The main tools used to prove Theorem \ref{T1} is the following modified Strichartz's estimate  for fourth-order Schr\"odinger equation, by which we can replace the spatial derivative of order $4\theta$ with the  time derivative of order $\theta$ in terms of Besov spaces.  For the definitions of the biharmonic admissible pairs set $\Lambda_b$,  and the Chemin--Lerner type space $l^2L^{\overline{q}}L^{\overline{r}}$, we refer to Section \ref{s2}.

\begin{proposition}\label{p1}
	Assume  $0<\theta<1$, $0\le\sigma<4\theta$, $(q,r),(\gamma,\rho)\in\Lambda_b$ are two  biharmonic admissible pairs, and  $\mu=0$ or $-1$. Assume also that   $1\le \overline{q}\le q$, $1\le  \overline{r}\le\infty$ satisfy  $\frac{4}{\overline{q}}-N\left(\frac{1}{2}-\frac{1}{\overline{r}}\right)=4(1-\theta)$.  Then for any $u_0\in H^{4\theta}$ and  $f\in B^{\theta}_{\gamma',2}(\R,L^{\rho'})\cap l^2L^{\overline{q}}\left(\R,L^{\overline{r}}\right)$, we have $e^{it(\Delta^2+\mu\Delta)}u_0, Gf\in C(\R,H^{4\theta})$
	where
	\begin{equation}
		(Gf)(t)=\int_0^te^{i(t-s)(\Delta^2+\mu\Delta)}f(s)ds.\notag
	\end{equation}
	Moreover, the following inequalities hold,
	\begin{equation}\label{i1}
		\|e^{it(\Delta^2+\mu\Delta)}u_0\|_{ L^qB^{4\theta}_{r,2}\cap B^{\theta-\sigma/4}_{q,2}{B^\sigma_{r,2}}}\lesssim \|u_0\|_{H^{4\theta}},
	\end{equation}
	\begin{equation}\label{i2}
		\|Gf\|_{ L^q B^{4\theta}_{r,2}}\lesssim \|f\|_{B^{\theta}_{\gamma',2}L^{\rho'}}+\|f\|_{l^2L^{\overline{q}}L^{\overline{r}}},
	\end{equation}
	\begin{equation}\label{i3}
		\|Gf\|_{B_{q, 2}^{\theta-\sigma/4}B_{r, 2}^{\sigma}} \lesssim\|f\|_{B_{\gamma^{\prime}, 2}^{\theta}L^{\rho^{\prime}}}+\|f\|_{l^2L^{\overline{q}}L^{\overline{r}}}.
	\end{equation}
\end{proposition}

In this paper, we first establish the modified Strichartz's estimate (\ref{i1})--(\ref{i3}) for the biharmonic Schr\"odinger equation in the spirit of \cite{Na,Na2}. Then we establish various nonlinear estimates and use  the contraction mapping principle based on  the modified Strichartz's estimate to complete the proof of Theorem \ref{T1}.

The rest of the paper is organized as follows. In Section \ref{s2}, we introduce some notations and give a review of the biharmonic Strichartz's estimates. In Section \ref{s3}, we establish the  modified Strichartz's estimate.  In Section \ref{s4}, we give the proof of Theorem \ref{T1}.
\section{Preliminary}\label{s2}
If $X, Y$ are nonnegative quantities, we sometimes use $X\lesssim Y$ to denote the estimate $X\leq CY$ for some positive constant $C$. Pairs of conjugate indices are written as $p$ and $p'$, where $1\leq p\leq\infty$ and $\frac1p+\frac1{p'}=1$. We use $L^p (\mathbb{R}^N)$ to denote the usual Lebesgue space and  $L^\gamma(I,L^\rho(\mathbb{R}^N))$ to denote the space-time Lebesgue spaces with the  norm
\begin{gather}\notag
	\|f\|_{L^\gamma(I,L^\rho(\mathbb{R}^N))}:=\left(\int_I\|f\|_{L_x^\rho}^\gamma dt\right)^{1/\gamma}
\end{gather}
for any time slab $I\subset\mathbb{R}^N$, with the usual modification when either $\gamma$ or $\rho$ is infinity.  We define the  Fourier transform on $\R,\R^N$ and $\R^{1+N}$   by
\begin{align*}
	&\hat f(\tau)=\int_\R f(t)e^{-it\tau}dt,\qquad\qquad\quad\qquad\qquad\tau\in\R,\\
	&\hat f(\xi)=\int_{\R^N}f(x)e^{-ix\cdot\xi}d\xi,\qquad\quad\qquad\qquad\xi\in\R^N,\\
	&\widetilde f(\tau,\xi)=\int_{\R^{1+N}}f(t,x) e^{-ix\cdot\xi-it\tau}dxdt,\quad(\tau,\xi)\in\R\times\R^N,
\end{align*}
respectively.

Next, we  review the definition of Besov spaces.  Let $\phi$ be a smooth function whose Fourier transform $\hat\phi$ is a non-negative even function which satisfies supp $\hat\phi\subset\{\tau\in\R,1/2\le|\tau|\le2\}$ and $\sum_{k=-\infty}^{\infty}\hat\phi(\tau/{2^k})=1$ for any $\tau\neq0$. For $k\in\mathbb{Z}$, we put  $\hat\phi_k(\cdot)=\hat\phi(\cdot/{2^k})$ and $\psi=\sum_{j=-\infty}^0\phi_j$.   Moreover, we define  $\chi_k=\sum_{k-2}^{k+2}\phi_j$ for $k\ge1$ and $\chi_0=\psi+\phi_1+\phi_2$.  For $s\in\R$ and $1 \leq p, q \leq \infty$, we define  the Besov space
$$
B_{p,q}^{s}\left({\R}^{N}\right)=\left\{u \in \mathcal{S}^{\prime}\left(\R^{N}\right), \|u\|_{B_{p,q}^{s}\left(\R^{N}\right)}<\infty\right\},
$$
where $\mathcal{S}^{\prime}\left({\R}^{N}\right)$ is the space of tempered distributions on $\R^{N},$ and
$$
\|u\|_{B_{p, q}^{s}\left(\R^{N}\right)}=\left\|\psi *_{x} u\right\|_{L^{p}\left(\R^{N}\right)}+\left\{\begin{array}{ll}
	\left\{\sum_{k \geq 1}\left(2^{s k}\left\|\phi_{k} *_{x} u\right\|_{L^{p}\left(\R^{N}\right)}\right)^{q}\right\}^{\frac1q}, & q<\infty, \\
	\sup _{k \geq 1} 2^{s k}\left\|\phi_{k}*_{x} u\right\|_{L^{p}\left(\R^{N}\right)}, & q=\infty,
\end{array}\right.
$$
where $*_{x}$ denotes the convolution with respect to the variables in $\R^{N}$.  Here we use $\phi_k*_xu$ to denote $\phi_k(|\cdot|)*_xu$. We also define $\chi_k*_xu,\psi*_xu,\chi_0*_xu$ similarly. This is an abuse of symbol, but no confusion is likely to arise.

For $1\le q$, $\alpha\le \infty$ and a Banach space $V$, we denote the Lebesgue space for functions on $\R$ to $V$ by $L^q\left(\R,V\right)$ and the Lorentz space by $L^{q,\alpha}\left(\R,V\right)$. We define the Sobolev space $H^{1,q}\left(\R,V\right)=\left\{u:u\in L^q\left(\R,V\right),\partial_tu\in L^q \left(\R,V\right)\right\}$.
For $1\le\alpha,r,q\le\infty$, we denote  the Chemin-Lerner type space $$l^{\alpha} L^{q}\left(\mathbb{R}, L^{r}\left(\mathbb{R}^{N}\right)\right)=\left\{u \in L_{\mathrm{loc}}^{1}\left(\mathbb{R}, L^{r}\left(\mathbb{R}^{N}\right)\right), \|u\|_{\ell^{\alpha} L^{q}\left(\mathbb{R},\  L^{r}\left(\mathbb{R}^{N}\right)\right)}<\infty\right\}$$ with the norm defined by
\[
\|u\|_{l^\alpha L^q(\R,L^r(\R^N))}=\|\psi*_xu\|_{L^q(\R,L^r(\R^N))}+\left(\sum_{k\ge1}\|\phi_k*_xu\|_{L^q(\R,L^r(\R^N))}^\alpha\right)^{1/\alpha}
\]
with trivial modification if $\alpha=\infty$. We also define $l^\alpha L^{q,\infty}\left(\R,L^r\left(\R^N\right)\right)$ similarly.
Finally, we  define the Besov space of vector-valued functions. Let $\theta \in\R, 1 \leq q, \alpha \leq \infty$ and $V$ be a Banach space. We put
$$
B_{q, \alpha}^{\theta}(\R, V)=\left\{u \in \mathcal{S}^{\prime}(\R, V) ;\|u\|_{B_{q, \alpha}^{\theta}(\R, V)}<\infty\right\}
$$
where
\begin{equation}\label{9221}
	\|u\|_{B_{q, \alpha}^{\theta}(\R, V)}=\left\|\psi *_{t} u\right\|_{L^{q}(\R , V)}+\left\{\sum_{k \geq 1}\left(2^{\theta k}\left\|\phi_{k} *_{t} u\right\|_{L^{q}(\R , V)}\right)^{\alpha}\right\}^{1 / \alpha}\notag
\end{equation}
with trivial modification if $\alpha=\infty.$ Here $*_{t}$ denotes the convolution in
$\R$.

In this paper, we  omit the integral domain for simplicity unless noted otherwise. For example, we write $l^\alpha L^qL^r=l^\alpha L^q\left(\R,L^r(\R^N)\right)$,  $L^qB^s_{r,2}=L^q\left(\R,B^s_{r,2}(\R^N)\right)$ and $B^{\theta-\sigma/4}_{q,2}B^\sigma_{r,2}=B^{\theta-\sigma/4}_{q,2}(\R,$ $B^\sigma_{r,2}(\R^N))$ etc.

Following standard notations, we introduce Schr\"odinger admissible pair as well as the corresponding Strichartz's estimate for the biharmonic Schr\"odinger equation.

\begin{definition}\label{bpair}
	A pair of Lebesgue space exponents $(\gamma, \rho)$ is called  biharmonic Schr\"odinger admissible for the equation (\ref{NLS}) if  $(\gamma, \rho)\in \Lambda_b$ where
	\begin{equation*}
		\Lambda_b=\{(\gamma, \rho):2\leq \gamma, \rho\leq\infty, \  \frac4\gamma+\frac N\rho=\frac N2, \  (\gamma, \rho, N)\neq(2, \infty, 4)\},
	\end{equation*}
\end{definition}
\begin{lemma}[Strichartz's estimate  for BNLS, \cite{Pa}]\label{L2.2S}
	Suppose that $(\gamma,\rho), (a,b)\in\Lambda_b $ are  two biharmonic admissible pairs, and  $\mu=0$ or $-1$. Then for any  $u\in L^2(\mathbb{R}^N)$ and $h\in L^{a'}(\R,L^{b'}(\mathbb{R}^N))$, we have
	\begin{gather}\label{sz}
		\|e^{it(\Delta^2+\mu\Delta)}u\|_{L^\gamma(\R, L^\rho)}\leq C\|u\|_{L^2},
	\end{gather}
	\begin{equation}\label{sz1}
		\|\int_{\R}e^{-is(\Delta^2+\mu\Delta)}h(s)\ ds \|_{L^2}\leq C\|h\|_{L^{a'}(\R,L^{b'})},
	\end{equation}
	\begin{equation}\label{SZ}
		\|\int_0^te^{i(t-s)(\Delta^2+\mu\Delta)}h(s)\ ds \|_{L^\gamma(\R, L^\rho)}\leq C\|h\|_{L^{a'}(\R,L^{b'})}.
	\end{equation}
\end{lemma}

\section{Modified Strichartz's estimate }\label{s3}
In this section, we  prove Proposition \ref{p1}. First, we prepare several lemmas. We assume the functions $\phi,\chi_0,\psi,\phi_j,\chi_j$ are defined in Section \ref{s2}.
\begin{lemma}\label{L1}
	Assume $N\ge1$, $\mu=0$ or $-1$, and  $K(t,x), K_j(t,x)(j\ge1):\R\times\R^N\rightarrow\mathbb{C}$ are defined by
	\begin{align*}
		K(t,x)=\fc1{(2\pi)^{1+N}}\int e^{it\tau+ix\cdot \xi}\fc{\hat\psi(|\xi|^4-\mu|\xi|^2)(1-\hat\chi_0(\tau))}{i(\tau-|\xi|^4+\mu|\xi|^2)} d\tau  d\xi,
	\end{align*}
	\begin{equation}
		K_j(t,x)=\fc1{(2\pi)^{1+N}}\int e^{it\tau+ix\cdot \xi}\fc{\hat\phi_j(|\xi|^4-\mu|\xi|^2)(1-\hat\chi_j(\tau))}{i(\tau-|\xi|^4+\mu|\xi|^2)} d\tau  d\xi.\notag
	\end{equation}
	Then for any $0<\theta<1$, $1\le q\le\infty$, $1\le r\le\infty$ with $\fc4{q}-N(1-\fc1{ r})=4\theta$, we have
	\begin{equation*}
		\|K\|_{L^{ q,1}L^{ r}}\le C,\quad\text{and  }\quad\|K_j\|_{L^{ q,1}L^{ r}}\le C2^{-j\theta},
	\end{equation*}
	where the constant $C$ independent of $j\ge1$.
\end{lemma}
\begin{proof}
	The method used here is inspired by the last part of  Lemma 2.4 in \cite{Wada}. We shall prove the estimate for $K_j(t,x)$, the estimate for $K(t,x)$ can be treated in a similar way.
	
	Define $\chi=\sum_{-2}^{2}\phi_j$  and
	\begin{equation}
		\widetilde {L_j}(\tau,\xi)=\fc{\hat\phi(|\xi|^4-\mu2^{- j/2}|\xi|^2)(1-\hat\chi(\tau))}{i(\tau-|\xi|^4+\mu2^{- j/2}|\xi|^2)},\qquad j\ge1.\notag
	\end{equation}
	Then by  Fourier transform
	$$
	\widetilde{K_j}(\tau,\xi)=\fc{\hat\phi_j(|\xi|^4-\mu|\xi|^2)(1-\hat\chi_j(\tau))}{i(\tau-|\xi|^4+\mu|\xi|^2)}
	= 2^{-j}\widetilde {L_j}(\tau/{2^j},\xi/{2^{ j/4}}),\notag
	$$
	so that $K_j(t,x)=2^{{Nj}/4}L_j(2^jt,2^{ j/4}x)$. Moreover, by the  change of variables, we have
	\begin{equation}
		\|K_j\|_{L^{q,1}L^{r}}=2^{j\left(\frac{N}{4}-\frac{1}{q}-\frac{N}{4r}\right)}\|L_j\|_{L^{q,1}L^{ r}}=2^{-j\theta}\|L_j\|_{L^{q,1}L^{ r}}.\notag
	\end{equation}
	Therefore,  it suffices  to show that for  any $l\ge1$, there exists $C>0$ independent of $j\ge1$ such that
	\begin{equation}\label{1172}
		|L_j(t,x)|\le C(1+|t|+|x|)^{-l},\qquad \forall (t,x)\in \R\times\R^N.
	\end{equation}

	We now prove (\ref{1172}). By Fourier inversion formula
	\begin{equation}
		L_j(t,x)=\fc1{(2\pi)^{1+N}}\int\int_{\R^{1+N}}e^{it\tau+ix\cdot\xi}\fc{\hat\phi(|\xi|^4-\mu2^{-j/2}|\xi|^2)(1-\hat\chi(\tau))}{i(\tau-|\xi|^4+\mu2^{- j/2}|\xi|^2)} d\tau d\xi.\notag
	\end{equation}
	Note that on the support of the integrand of $L_{j},$ we must have $|\tau| \notin[1 / 4,4]$ and $\left||\xi|^4-\mu2^{-j/2}|\xi|^{2}\right| \in[1 / 2,2],$ so that $\left|\tau-|\xi|^4+\mu2^{-j/2}|\xi|^2\right| \geq 1/4$. Therefore, we deduce that the following integral
	\[
	\int_{\left|\tau-|\xi|^4+\mu2^{-j/2}|\xi|^2\right|  \leq 10} e^{it\tau+ix\cdot\xi}\fc{\hat\phi(|\xi|^4-\mu2^{- j/2}|\xi|^2)(1-\hat\chi(\tau))}{i(\tau-|\xi|^4+\mu2^{- j/2}|\xi|^2)} d\tau d\xi\]
	is bounded. On the other hand, since $\hat{\chi}(\tau)=0$ when $\left|\tau-|\xi|^4+\mu2^{- j/2}|\xi|^2\right|\ge10$ and $1/2\le \left||\xi|^4-\mu2^{- j/2}|\xi|^{2}\right|\le2,$ we have
	\begin{eqnarray*}
		&&\int_{\left|\tau-|\xi|^4+\mu2^{- j/2}|\xi|^2\right|\geq 10} e^{i t \tau} \fc{\hat\phi(|\xi|^4-\mu2^{- j/2}|\xi|^2)(1-\hat\chi(\tau))}{i(\tau-|\xi|^4+\mu2^{- j/2}|\xi|^2)} d\tau\\
		&=&\int_{\left|\tau-|\xi|^4+\mu2^{- j/2}|\xi|^2\right|\geq 10} e^{i t \tau} \fc{\hat\phi(|\xi|^4-\mu2^{- j/2}|\xi|^2)}{i(\tau-|\xi|^4+\mu2^{- j/2}|\xi|^2)} d\tau\\
		&=&2 \operatorname{sign}(t) e^{i t(|\xi|^4-\mu2^{- j/2}|\xi|^{2})} \hat\phi(|\xi|^4-\mu2^{- j/2}|\xi|^2) \int_{10/|t|}^{\infty} \frac{\sin \tau}{\tau} d \tau.
	\end{eqnarray*}
	This is also bounded, so that  we have proved the boundedness of $L_{j}(t, x)$.
	
	Moreover, for $1 \leq l \leq N,$ the integration by parts shows that
	\begin{eqnarray}\label{1194}
		&&x_{l} L_{j}(t, x)\notag\\
		&=&\frac{1}{(2 \pi)^{1+N}} \iint_{R^{1+N}} e^{i t \tau+i x\cdot \xi} \frac{\partial}{\partial \xi_{l}} \fc{\hat\phi(|\xi|^4-\mu2^{-j/2}|\xi|^2)(1-\hat\chi(\tau))}{\tau-|\xi|^4+\mu2^{- j/2}|\xi|^2} d\tau d\xi\notag\\
		&=&\frac{1}{(2 \pi)^{1+N}} \iint_{R^{1+N}} e^{i t \tau+i x\cdot \xi} (4|\xi|^2\xi_{l}-2\mu2^{- j/2}\xi_{l})\left\{\frac{\hat{\phi}^{\prime}\left(|\xi|^4-\mu2^{- j/2}|\xi|^{2}\right)\left(1-\hat{\chi}(\tau)\right)}{\tau-|\xi|^4+\mu2^{- j/2}|\xi|^{2}}\right.\notag\\
		&&+\left.\fc{\hat\phi(|\xi|^4-\mu2^{-j/2}|\xi|^2)(1-\hat\chi(\tau))}{(\tau-|\xi|^4+\mu2^{- j/2}|\xi|^2)^2}\right\} d \tau d \xi.
	\end{eqnarray}
	The right-hand side of (\ref{1194}) is bounded as before. Similarly, $tL_j(t,x)$ is also bounded.  Repeating this, we can obtain the desired estimate (\ref{1172}).
\end{proof}

\begin{lemma}\label{f}
	Let  $N\ge1$, $0<\theta<1$, $1\le r_0$, $\overline{r}$, $\overline{q}$, $\gamma \le \infty$, $1<q_0$, $\rho<\infty$. Assume that  $2\le\overline r\le\infty,\fc4{\overline q}-N(\fc12-\fc1{\overline r})=\frac{4}{q_0}-N(\frac{1}{2}-\frac{1}{r_0})=4(1-\theta),(\gamma,\rho)\in\Lambda_b$ and $r_0$ satisfies $\rho'\le r_0<\overline{r}$ or $\overline{r}<r_0\le\rho'$.  Then for any  $f\in l^2L^{\overline q}L^{\overline r}\cap B^{\theta}_{\gamma',2}L^{\rho'}$, we have
	\begin{equation}
		\|f\|_{l^2L^{q_0,\infty}L^{r_0}}\lesssim \|f\|_{l^2L^{\overline q}L^{\overline r}}+\|f\|_{B^{\theta}_{\gamma',2}L^{\rho'}}.\notag
	\end{equation}
\end{lemma}
\begin{proof}
	The proof is an obvious adaptation of Lemma 2.5 in \cite{Na}.
\end{proof}
\begin{lemma}\label{dj}
	Let $s\in\R$, $1\le p,q\le\infty$, $\mu=0$ or $-1$, then the norm defined by
	\begin{eqnarray*}
		&&\|u\|_{\widetilde{B }_{p, q}^{s}\left(\mathbb{R}^{n}\right)}:=\left\|\left(\mathcal{F}_{\xi}^{-1}\left(\widehat{\psi}\left(|\xi|^{4}-\mu|\xi|^2\right)\right)\right) *_{x} u\right\|_{L^{p}\left(\mathbb{R}^{n}\right)} \\
		&&+\begin{cases}\left\{\sum_{j \geq 1}\left(2^{s j / 4}\left\|\left(\mathcal{F}_{\xi}^{-1}\left(\widehat{\phi}\left((|\xi|^{4}-\mu|\xi|^2)/{2^{j}}\right)\right)\right) *_{x} u\right\|_{L^{p}\left(\mathbb{R}^{n}\right)}\right)^{q}\right\}^{\frac1q},\ \text{ if }q<\infty,\\
			\sup _{j \geq 1} 2^{s j / 4}\left\|\left(\mathcal{F}_{\xi}^{-1}\left(\widehat{\phi}\left((|\xi|^{4}-\mu|\xi|^2)/{2^{j}}\right)\right)\right) *_{x} u\right\|_{L^{p}\left(\mathbb{R}^{n}\right)},\ \text{ if }q=\infty,\end{cases}
	\end{eqnarray*}
	is equivalent to the norm $\|u\|_{B^s_{p,q}}(\R^N)$ for any function $u$.
\end{lemma}
\begin{proof}
	For the sake of convenience and completeness, we briefly sketch the proof. Indeed, readers seeking a fuller treatment of certain details may consult Lemma 2.3 of \cite{Na}.
	
	Firstly, we show that $\|u\|_{B_{p,q}^{s}\left(\mathbb{R}^{N}\right)} \lesssim\|u\|_{\widetilde{B }_{p,q}^{s}\left(\mathbb{R}^{N}\right)}.$  Since
	$$
	\sum_{k=-6}^7 \widehat{\phi}\left((|\xi|^{4}-\mu|\xi|^2)/{2^{4j+k}}\right)=1
	$$
	on the support of $\widehat{\phi}\left(|\cdot|/2^j\right),$ it follows from  Young's inequality that
	\begin{eqnarray*}
		2^{js}\left\|\phi_{j} *_{x} u\right\|_{L^{p}\left(\mathbb{R}^{N}\right)} &=&2^{js}\left\|\left(\mathcal{F}_{\xi}^{-1} \widehat{\phi}\left({|\xi|}/{2^{j}}\right)\right) *_{x} u\right\|_{L^{p}\left(\mathbb{R}^{N}\right)} \\
		& \lesssim& 2^{js}\sum_{k=-4}^{5}\left\|\left(\mathcal{F}_{\xi}^{-1}\left(\widehat{\phi}\left((|\xi|^{4}-\mu|\xi|^2)/{2^{4j+k}}\right)\right)\right) *_{x} u\right\|_{L^{p}\left(\mathbb{R}^{N}\right)}\\
		&\lesssim& \sum_{k=-4}^{5}2^{\fc{ls}4}\left\|\left(\mathcal{F}_{\xi}^{-1}\left(\widehat{\phi}\left((|\xi|^{4}-\mu|\xi|^2)/{2^{l}}\right)\right)\right) *_{x} u\right\|_{L^{p}\left(\mathbb{R}^{N}\right)},
	\end{eqnarray*}
	where $l=4j+k$. Similar inequality for  the low frequency part also holds.  Taking  $l^q(\mathbb{Z})$ norm,   we  obtain the desired inequality $\|u\|_{B_{p,q}^{s}\left(\mathbb{R}^{N}\right)} \lesssim\|u\|_{\widetilde{B}_{p,q}^{s}\left(\mathbb{R}^{N}\right)}$.
	
	Next, we show the opposite inequality. Note that  $\sum_{k=-4}^{4}\hat\phi({|\xi|}/{2^{[ j/4]+k}})=1$ on the support of $\hat\phi((|\xi|^{4}-\mu|\xi|^2)/{2^j})$, so that
	\begin{eqnarray*}
		&&2^{{js}/4}\left\|\mathcal{F}_{\xi}^{-1}\left(\widehat{\phi}\left((|\xi|^{4}-\mu|\xi|^2)/{2^{j}}\right)\right) *_{x} u\right\|_{L^{p}\left(\mathbb{R}^{N}\right)} \\
		&\lesssim& 2^{{js}/4}\sum_{k=-2}^{2}\left\|\mathcal{F}_{\xi}^{-1}\left(\widehat{\phi}\left({|\xi|}/{2^{[ j/4]+k}}\right)\right) *_{x} u\right\|_{L^{p}\left(\mathbb{R}^{N}\right)}\\
		&\lesssim& \sum_{k=-2}^{2}2^{ls}\left\|\mathcal{F}_{\xi}^{-1}\left(\widehat{\phi}\left(|\xi|/{2^{l}}\right)\right) *_{x} u\right\|_{L^{p}\left(\mathbb{R}^{N}\right)},
	\end{eqnarray*}
	where $l=[ j/4]+k$. Since the low frequency parts are easier to treat, we can take  $l^q(\mathbb{Z})$ norm to  obtain the desired inequality $\|u\|_{\widetilde{B}_{p,q}^{s}\left(\mathbb{R}^{N}\right)}\lesssim \|u\|_{B_{p,q}^{s}\left(\mathbb{R}^{N}\right)}.$
\end{proof}

\begin{proof}[\textbf{Proof of Proposition \ref{p1}}] We use the natation $\phi_{ j/4}=\mathcal{F}_\xi^{-1}\left(\hat\phi_j(|\xi|^4-\mu|\xi|^2)\right)$. This is an abuse of symbol, but no confusion is likely to arise. Under this notation, we obtain the following equivalence from Lemma \ref{dj},
	\begin{equation}\label{1271}
		\|u\|_{{B}_{p, q}^{s}}\approx\left\|\left(\mathcal{F}_{\xi}^{-1}\left(\widehat{\psi}\left(|\xi|^{4}-\mu|\xi|^2\right)\right)\right) *_{x} u\right\|_{L^{p}}
		+\left\{\sum_{j=1}^\infty\left(2^{s j / 4}\left\| \phi_{ j/4}*_{x} u\right\|_{L^{p}}\right)^{q}\right\}^{\frac1q}
	\end{equation}
	with trivial modification if $q=\infty$. Then we claim that for any $f:\R^N\rightarrow \mathbb{C}$, we have
	\begin{equation}\label{a1}
		\phi_{j} *_{t}e^{it(\Delta^2+\mu\Delta)}f=e^{it(\Delta^2+\mu\Delta)}\phi_{ j/4} *_{x} f.
	\end{equation}
	In fact, by Fourier transform
	\begin{eqnarray}
		\left(\phi_j*_te^{it(\Delta^2+\mu\Delta)}f\right)\hat{\phantom{f}}(t,\xi)&=&\int_{-\infty }^{\infty } \phi_j(\tau)e^{i(t-\tau)\left(\left|\xi\right|^{4}-\mu \left|\xi\right|^{2}\right)}\hat f(\xi)\mathrm{d}\tau \notag\\
		&=& e^{it \left(\left|\xi\right|^{4}-\mu \left|\xi\right|^{2}\right)}\hat f(\xi)\hat\phi_j \left(\left|\xi\right|^{4}-\mu \left|\xi\right|^{2}\right).\notag
	\end{eqnarray}
	Taking Fourier inverse transform, we obtain  (\ref{a1}).
	
	We now resume the proof of Proposition \ref{p1}. We seperate the proof into three parts.\\
	\textbf{The proof of the inequality (\ref{i1})}.  Using the same method as that used to prove Corollary 2.3.9 in \cite{Ca9}, we deduce that $e^{it(\Delta^2+\mu\Delta)}u_0\in C\left(\R,H^{4\theta }\right)$ and
	\begin{equation}\label{1166}
		\left\|e^{it(\Delta^2+\mu\Delta)}u_0\right\|_{L^qB^{4\theta}_{r,2}}\lesssim \left\|u_0\right\|_{H^{4\theta}}.
	\end{equation}
	It remains to estimate $\left\|e^{it(\Delta^2+\mu\Delta)}u_0\right\|_{B^{\theta-\sigma/4}_{q,2}B^\sigma_{r,2}}$.
	Applying (\ref{a1}) and Strichartz's estimate (\ref{sz}), we conclude that
	\begin{eqnarray}
		&&\sum_{j=1}^{\infty}  \sum_{k=1}^{\infty} 2^{(2 \theta-\sigma/2) j+\sigma k/2}\left\|\phi_{j} *_{t} \phi_{ k/4} *_{x} e^{it(\Delta^2+\mu\Delta)}u_0\right\|_{L^{q}L^{r}}^{2} \notag\\
		& \lesssim& \sum_{j=1}^{\infty} \sum_{k=1}^{\infty} 2^{(2 \theta-\sigma/2) j+\sigma k/2 }\left\|\phi_{ j/4} *_{x} \phi_{k/4} * u_0\right\|_{L^{2}}^{2} \notag\\
		& \lesssim& \sum_{k=1}^{\infty} 2^{2 \theta k}\left\|\phi_{ k/4} *_{x} u_0\right\|_{L^{2}}^{2} \lesssim\left\|u_0\right\|_{H^{4 \theta}}^{2},\notag
	\end{eqnarray}
	where we  used (\ref{1271}) and the fact $\hat{\phi}_{j}\left(|\xi|^4-\mu|\xi|^{2}\right) \hat{\phi}_{k}\left(|\xi|^4-\mu|\xi|^{2}\right)=0$ whenever $|j-k|\ge 2$. Since the low frequency parts are easier to treat, we obtain
	$$
	\left\|e^{it(\Delta^2+\mu\Delta)}u_0\right\|_{B_{q, 2}^{\theta-\sigma/4}B_{r, 2}^{\sigma}} \lesssim\left\|u_0\right\|_{H^{4 \theta}}.
	$$
	This inequality together with (\ref{1166}) yields (\ref{i1}).

	\noindent\textbf{The proof of the inequality (\ref{i2})}.
	Taking Fourier transform, we get
	\begin{eqnarray}\label{9206}
		(Gf)\hat{\phantom{f}}(t,\xi)&=&\int_0^te^{i(t-s)(|\xi|^4-\mu|\xi|^2)}\hat f(s,\xi)ds \notag\\
		&=&\fc1{2\pi}\int_0^te^{i(t-s)(|\xi|^4-\mu|\xi|^2)}ds \int_{-\infty}^{\infty}\widetilde f(\tau,\xi)e^{i\tau s}d\tau\notag\\
		&=&\int_{-\infty}^{\infty}\fc{e^{it\tau}-e^{it(|\xi|^4-\mu|\xi|^2)}}{2\pi i(\tau-|\xi|^4+\mu|\xi|^2)}\widetilde f(\tau,\xi)d \tau.
	\end{eqnarray}
	We then multiply  (\ref{9206}) with $\hat\phi_k(|\xi|^4-\mu|\xi|^2)$ to obtain
	\begin{eqnarray}\label{1272}
		&&\hat\phi_k(|\xi|^4-\mu|\xi|^2)(Gf)\hat{\phantom{f}}(t,\xi)\notag\\
		&=&\int_{-\infty}^{\infty}\fc{e^{it\tau}\hat\phi_k(|\xi|^4-\mu|\xi|^2)\hat\chi_k(\tau)}{{2\pi i(\tau-|\xi|^4+\mu|\xi|^2)}}\widetilde f(\tau,\xi) d\tau\notag\\
		&&+\int_{-\infty}^{\infty}\fc{e^{it\tau}\hat\phi_k(|\xi|^4-\mu|\xi|^2)(1-\hat\chi_k(\tau))}{2\pi i(\tau-|\xi|^4+\mu|\xi|^2)}\cdot \hat\chi_k(|\xi|^4-\mu|\xi|^2)\widetilde f(\tau,\xi) d\tau\notag\\
		&&-\int_{-\infty}^{\infty}\fc{e^{it(|\xi|^4-\mu|\xi|^2)}}{{2\pi i(\tau-|\xi|^4+\mu|\xi|^2)}}\hat\phi_k(|\xi|^4-\mu|\xi|^2)\hat\chi_k(\tau)\widetilde f(\tau,\xi) d\tau\notag\\
		&&-\int_{-\infty}^{\infty}\fc{e^{it(|\xi|^4-\mu|\xi|^2)}\hat\phi_k(|\xi|^4-\mu|\xi|^2)(1-\hat\chi_k(\tau))}{2\pi i(\tau-|\xi|^4+\mu|\xi|^2)}\cdot \hat\chi_k(|\xi|^4-\mu|\xi|^2)\widetilde f(\tau,\xi) d\tau
	\end{eqnarray}
	where we used the face that $\hat\chi_k=1$ on the support of $\hat\phi_k$. Since
	\begin{eqnarray*}
		\mathcal{F}_\tau^{-1}\{\fc1{i(\tau-|\xi|^4+\mu|\xi|^2)}\}(t)&=&\fc1{2\pi}\int_{-\infty}^{\infty}e^{it\tau}\fc1{i(\tau-|\xi|^4+\mu|\xi|^2)}d\tau\\
		&=&\fc{1}{2\pi}e^{it(|\xi|^4-\mu|\xi|^2)}\int_{-\infty}^{\infty}\fc{e^{it\tau}}{i\tau} d\tau\\
		&=&\fc12\text{sign}(t)e^{it(|\xi|^4-\mu|\xi|^2)},
	\end{eqnarray*}
	it follows  from (\ref{1272}) that
	\begin{eqnarray}\label{1168}
		\phi_{ k/4}*_x(Gf)&=&\fc12\int_{-\infty}^{\infty}\text{sign}(t-s)e^{i(t-s)(\Delta^2+\mu\Delta)}(\phi_{ k/4}*_x\chi_k*_tf)(s)ds\notag\\
		&&+K_k*_{t,x}\chi_{ k/4}*_xf\notag\\
		&&-\fc12e^{it(\Delta^2+\mu\Delta)}\int_{-\infty}^{\infty}\text{sign}(-s)e^{is(\Delta^2+\mu\Delta)}(\phi_{ k/4}*_x\chi_k*_tf)(s) ds \notag\\
		&&-e^{it(\Delta^2+\mu\Delta)}\{K_k*_{t,x}\chi_{k/4}*_xf\}|_{t=0},
	\end{eqnarray}
	where $K_j$ is the function  defined in Lemma \ref{L1}. Applying (\ref{1168}), Strichartz's estimates (\ref{sz})--(\ref{SZ}), we conclude that
	\begin{equation}\label{9201}
		\|\phi_{k/4}*_x(Gf)\|_{L^qL^r} \lesssim \|\chi_k*_tf\|_{L^{\gamma'}L^{\rho'}}+\|K_k*_{t,x}\chi_{ k/4}*_xf\|_{L^qL^r\cap L^\infty L^2}.
	\end{equation}
	Next, we estimate $\|K_k*_{t,x}\chi_{ k/4}*_xf\|_{L^qL^r\cap L^\infty L^2}$. Let
	\begin{gather}\notag
		r_0=\begin{cases}\overline r,\qquad\text{if }1\le\overline r\le2,\\
			2, \qquad\text{if }\overline r\ge2,\end{cases} \qquad r_1=\begin{cases}\overline r,\qquad\text{if }1\le\overline r\le r,\\
			2, \qquad\text{if }\overline r\ge r.\end{cases}
	\end{gather}
	We then define $q_0,\widetilde{r_0},\widetilde{q_0}$, $q_1,\widetilde{r_1},\widetilde{q_1}$ such that $\fc4{q_0}-N(\fc12-\fc1{r_0})=4(1-\theta)$, $1+\fc12=\fc1{\widetilde{r_0}}+\fc1{r_0},1=\fc1{\widetilde {q_0}}+\fc1{q_0}$ and  $\fc4{q_1}-N(\fc12-\fc1{r_1})=4(1-\theta)$, $1+\fc1r=\fc1{\widetilde{r_1}}+\fc1{r_1},1+\fc1q=\fc1{\widetilde {q_1}}+\fc1{q_1}$. Then it is easy to check that $1\le r_0,q_0,\widetilde{r_0},\widetilde{q_0}$, $r_1, q_1,\widetilde{r_1},\widetilde{q_1}\le \infty $ and  $\fc4{\widetilde{q_0}}-N(1-\fc1{\widetilde{r_0}})=\fc4{\widetilde{q_1}}-N(1-\fc1{\widetilde{r_1}})=4\theta$. From Young's inequality and Lemma \ref{L1}, we have
	\begin{equation}\label{1167}
		\|K_k*_{t,x}\chi_{ k/4}*_xf\|_{L^\infty L^2}\lesssim \|K_k\|_{L^{\widetilde{q_0 },1}L^{\widetilde{r_0 }}}\|\chi_{ k/4}*_xf\|_{L^{q_0,\infty} L^{r_0}}\lesssim2^{-k\theta}\|\chi_{ k/4}*_xf\|_{L^{q_0,\infty} L^{r_0}},
	\end{equation}
	and
	\begin{equation}\label{9213}
		\|K_k*_{t,x}\chi_{ k/4}*_xf\|_{L^q L^r}\lesssim \|K_k\|_{L^{\widetilde {q_1},1}L^{\widetilde {r_1}}}\|\chi_{ k/4}*_xf\|_{L^{q_1,\infty} L^{r_1}}\lesssim2^{-k\theta}\|\chi_{ k/4}*_xf\|_{L^{q_1,\infty} L^{r_1}}.
	\end{equation}	
	Estimates (\ref{9201}), (\ref{1167}) and (\ref{9213}) imply
	\begin{equation}\label{1241}
		\left\|\phi_{k/4}*_x\left(Gf\right)\right\|_{L^qL^r}\lesssim \left\|\chi_k*_tf\right\|_{L^{\gamma '}L^{\rho'}}+2^{-k\theta }\left\|\chi_{k/4}*_xf\right\|_{L^{q_0,\infty} L^{r_0}\cap L^{q_1,\infty }L^{r_1}}.
	\end{equation}
	Similarly,
	\begin{eqnarray}\label{9207}
		&&\|\mathcal{F}_\xi^{-1}(\hat\psi(|\xi|^4-\mu|\xi|^2))*_x(Gf)\|_{L^qL^r\cap L^\infty L^2}\notag\\
		&\lesssim&\|\chi_0*_tf\|_{L^{\gamma'}L^{\rho'}}+\|\mathcal{F}^{-1}_\xi\left(\hat\chi_0(|\xi|^4-\mu|\xi|^2)\right)*_xf\|_{L^{q_0,\infty} L^{r_0}\cap L^{q_1,\infty} L^{r_1}}.
	\end{eqnarray}
	It now follows from  (\ref{1241}), (\ref{9207}) and (\ref{1271}) that
	$$
	\|Gf\|_{ L^{q}B^{4\theta}_{r,2}} \lesssim \|f\|_{B^{\theta}_{\gamma',2}L^{\rho'}}+\|f\|_{l^2L^{q_0,\infty} L^{r_0}}+\|f\|_{l^2 L^{q_1,\infty} L^{r_1}}
	\lesssim \|f\|_{B^{\theta}_{\gamma',2}L^{\rho'}}+\|f\|_{l^2L^{\overline q} L^{\overline r}},
	$$
	where we used Lemma \ref{f} when $\overline{r}\ge2$.
	This proves (\ref{i2}).
	
	\noindent\textbf{The proof of the inequality (\ref{i3})}.	By the definition of Besov norm and (\ref{1271}), we have
	\begin{equation}\label{1192}
		\|Gf\|_{B_{q, 2}^{\theta-\sigma/4}B_{r, 2}^{\sigma}} \lesssim\|Gf\|_{L^{q}B_{r, 2}^{\sigma}}+\|Gf\|_{B_{q, 2}^{\theta}L^{r}}+J,
	\end{equation}
	where $J=\left\{\sum_{j=1}^{\infty} \sum_{k=1}^{\infty} 2^{(2 \theta-\sigma/2) j+\sigma k/2}\left\|\phi_{j} *_{t} \phi_{ k/4} *_{x} (Gf)\right\|_{L^{q}L^{r}}^2\right\}^{{1}/{2}}$.
	Since $\|Gf\|_{L^{q}B_{r, 2}^{\sigma}}$ can be controlled by (\ref{i2}), it suffices to estimate the last two terms in (\ref{1192}).

	We first estimate $\|Gf\|_{B_{q, 2}^{\theta}L^{r}}$. Since  $\phi_j*_te^{ita}=e^{ita}\hat\phi_j(a)$, for any $a\in\R$,  it follows from (\ref{9206}) that
	\begin{eqnarray}
		\phi_j*_t(\widehat{Gf})
		&=&\int_{-\infty}^{\infty}\fc{e^{it\tau}\hat\phi_j(\tau)}{2\pi i(\tau-|\xi|^4+\mu|\xi|^2)}\widetilde f(\tau,\xi) d\tau\notag\\
		&&-\int_{-\infty}^{\infty}\fc{e^{it(|\xi|^4-\mu|\xi|^2)}\hat\phi_j(|\xi|^4-\mu|\xi|^2)\hat\chi_j(\tau)}{{2\pi i(\tau-|\xi|^4+\mu|\xi|^2)}}\widetilde f(\tau,\xi) d\tau\notag\\
		&&-\int_{-\infty}^{\infty}\fc{e^{it(|\xi|^4-\mu|\xi|^2)}\hat\phi_j(|\xi|^4-\mu|\xi|^2)(1-\hat\chi_j(\tau))}{2\pi i(\tau-|\xi|^4+\mu|\xi|^2)}\notag\\
		&&\qquad\quad\cdot \hat\chi_j(|\xi|^4-\mu|\xi|^2)\widetilde f(\tau,\xi) d\tau,\notag
	\end{eqnarray}
	so that
	\begin{eqnarray}\label{1191}
		\phi_j*_t(Gf)&=&\fc12\int_{-\infty}^{\infty}\text{sign}(t-s)e^{i(t-s)(\Delta^2+\mu\Delta)}(\phi_j*_tf)(s)ds\notag\\
		&&-\fc12e^{it(\Delta^2+\mu\Delta)}\int_{-\infty}^{\infty}\text{sign}(-s)e^{is(\Delta^2+\mu\Delta)}(\phi_{j/4}*_x\chi_j*_tf)(s) ds\notag \\
		&&-\fc12e^{it(\Delta^2+\mu\Delta)}\{K_j*_{t,x}\chi_{ j/4}*_xf\}|_{t=0}. \notag
	\end{eqnarray}
	This together with  Strichartz's estimates (\ref{sz})--(\ref{SZ}) and (\ref{1167}) implies
	\begin{equation}	\|\phi_j*_t(Gf)\|_{L^qL^r}\lesssim\|\phi_j*_tf\|_{L^{\gamma'}L^{\rho'}}+\|\chi_j*_tf\|_{L^{\gamma'}L^{\rho'}}+2^{-j\theta}\left\|\chi_{j/4}*_xf\right\|_{L^{q_0,\infty }L^{r_0}}.\notag
	\end{equation}
	Similarly,
	\begin{equation}
		\|\psi*_t(Gf)\|_{L^qL^r}\lesssim\|\psi*_tf\|_{L^{\gamma'}L^{\rho'}}+\|\chi_0*_tf\|_{L^{\gamma'}L^{\rho'}}+\|\mathcal{F}^{-1}_\xi \left\{\hat\chi_0(|\xi|^4-\mu|\xi|^2)\right\}*_xf\|_{L^{q_0,\infty} L^{r_0}}.\notag
	\end{equation}
	Combining the  above two inequalities, we obtain
	$$
	\|Gf\|_{B^{\theta}_{q,2}L^{r}} \lesssim \|f\|_{B^{\theta}_{\gamma',2}L^{\rho'}}+\|f\|_{l^2L^{q_0,\infty} L^{r_0}}
	\lesssim \|f\|_{B^{\theta}_{\gamma',2}L^{\rho'}}+\|f\|_{l^2L^{\overline q} L^{\overline r}},
	$$
	where we used Lemma \ref{f} when $\overline{r}\ge2$.
	
	Next, we estimate $J$. Note that by (\ref{1168}) and (\ref{a1})
	\begin{eqnarray*}
		&&	\phi_j*_t\phi_{k/4}*_x(Gf)\\
		&=&\fc12\int_{-\infty}^{\infty}\text{sign}(t-s)e^{i(t-s)(\Delta^2+\mu\Delta)}(\phi_j*_t\phi_{ k/4}*_x\chi_k*_tf)(s)ds\\
		&&+K_k*_{t,x}\phi_j*_t\chi_{ k/4}*_xf\\
		&&-\fc12e^{it(\Delta^2+\mu\Delta)}\int_{-\infty}^{\infty}\text{sign}(-s)e^{is(\Delta^2+\mu\Delta)}(\phi_{ j/4}*_x\phi_{ k/4}*_x\chi_k*_tf)(s) ds \\
		&&-e^{it(\Delta^2+\mu\Delta)}\{K_k*_{t,x}\phi_{ j/4}*_x\chi_{k/4}*_xf\}|_{t=0},\notag\\
		&:=&\uppercase\expandafter{\romannumeral1}_{j,k}+\uppercase\expandafter{\romannumeral2}_{j,k}+\uppercase\expandafter{\romannumeral3}_{j,k}+\uppercase\expandafter{\romannumeral4}_{j,k},
	\end{eqnarray*}
	so that
	\begin{eqnarray*}
		J &\lesssim &\left(\sum_{j=1}^{\infty} \sum_{k=1}^{\infty} 2^{(2 \theta-\sigma/2) j+\sigma k/2}\left\|\uppercase\expandafter{\romannumeral1}_{j,k}\right\|_{L^{q}L^{r}}^{2}\right)^{\frac12}
		+\left(\sum_{j=1}^{\infty} \sum_{k=1}^{\infty} 2^{(2 \theta-\sigma/2) j+\sigma k/2}\left\|\uppercase\expandafter{\romannumeral2}_{j,k}\right\|_{L^{q}L^{r}}^{2}\right)^{\frac12}\notag\\
		&&+\left(\sum_{j=1}^{\infty} \sum_{k=1}^{\infty} 2^{(2 \theta-\sigma/2) j+\sigma k/2}\left\|\uppercase\expandafter{\romannumeral3}_{j,k}\right\|_{L^{q}L^{r}}^{2}\right)^{\frac12} +\left(\sum_{j=1}^{\infty} \sum_{k=1}^{\infty} 2^{(2 \theta-\sigma/2) j+\sigma k/2}\left\|\uppercase\expandafter{\romannumeral4}_{j,k}\right\|_{L^{q}L^{r}}^{2}\right)^{\frac12}\notag\\
		&:=& J_{1}+J_{2}+J_{3}+J_{4}.\notag
	\end{eqnarray*}
	We first consider  $J_1,J_3$. By Strichartz's estimates  (\ref{sz})-(\ref{SZ}), we get
	\begin{equation}
		\|\uppercase\expandafter{\romannumeral1}_{j,k}\|_{L^qL^r} \lesssim \|\phi_{j}*_t\chi_k*_tf\|_{L^{\gamma'}L^{\rho'}}
		,\quad\|\uppercase\expandafter{\romannumeral3}_{j,k}\|_{L^qL^r}\lesssim\|\phi_{ j/4}*_x\phi_{ k/4}*_x\chi_k*_tf\|_{L^{\gamma'}L^{\rho'}}.\notag
	\end{equation}
	Since  $\phi_j*_t\chi_k=0$ whenever $|j-k|\ge4$, we have
	\begin{eqnarray}
		J_1^2&\lesssim&\sum_{j=1}^\infty\sum_{k=1}^\infty2^{(2\theta-\sigma/2)j+\sigma/2 k}\|\phi_{j}*_t\chi_k*_tf\|_{L^{\gamma'}L^{\rho'}}^2\notag\\
		&\lesssim&\sum_{j=1}^{\infty}2^{2\theta j}\|\phi_j*_tf\|_{L^{\gamma'}L^{\rho'}}^2\lesssim\|f\|_{B^\theta_{\gamma',2}L^{\rho'}}^2.\notag
	\end{eqnarray}
	Similarly, we have
	$J_3\lesssim\|f\|_{B^\theta_{\gamma',2}L^{\rho'}}.$
	
	For $J_4$, we deduce from (\ref{SZ}) and  (\ref{1167}) that
	\begin{equation}
		\|\uppercase\expandafter{\romannumeral4}_{j,k}\|_{L^qL^r}\lesssim2^{-k\theta}\|\phi_{ j/4}*_x\chi_{ k/4}*_xf\|_{L^{q_0,\infty} L^{r_0}}.\notag
	\end{equation}
	Since $\phi_{ j/4}*_x\chi_{ k/4}=0$ whenever $|j-k|\ge4$,  we conclude that
	\begin{equation}
		J_4\lesssim\|f\|_{l^2L^{q_0,\infty}L^{r_0}}\lesssim \|f\|_{B^{\theta}_{\gamma',2}L^{\rho'}}+\|f\|_{l^2L^{\overline q} L^{\overline r}},\notag
	\end{equation}
	where we used Lemma \ref{f} when $\overline{r}\ge2$.
	
	Our final step is to estimate $J_2$.
	Similarly to (\ref{9213}), we have
	\begin{equation}\label{1169}
		\|\uppercase\expandafter{\romannumeral2}_{j,k}\|_{L^qL^r} \lesssim 2^{-k\theta }\|\phi_j*_t\chi_{ k/4}*_xf\|_{L^{q_1,\infty}L^{r_1}}.
	\end{equation}
	On the other hand, by Young's inequality
	\begin{equation}\label{11610}
		\|\uppercase\expandafter{\romannumeral2}_{j,k}\|_{L^qL^r} \lesssim \|\phi_j*_tf\|_{L^{\gamma'}L^{\rho'}}.
	\end{equation}
	It follows from (\ref{1169}) and (\ref{11610}) that
	\begin{eqnarray*}
		J_{2} &\lesssim &\left(\sum_{j=1}^{\infty} \sum_{k=1}^{j} 2^{(2\theta-\sigma/2) j+\sigma k/2}\left\|\phi_{j} *_{t} f\right\|_{L^{\gamma'}L^{\rho'}}^{2}\right)^{1/2} \\
		&&+\left(\sum_{j=1}^{\infty} \sum_{k=j+1}^{\infty} 2^{(2 \theta-\sigma/2) (j-k)}\left\|\phi_{j} *_{t} \chi_{ k/4} *_{x} f\right\|_{L^{q_1,\infty}L^{r_1}}^{2}\right)^{1/2} \\
		&:=& J_{2,1}+J_{2,2}.
	\end{eqnarray*}
	Since $\sum_{k=1}^{j} 2^{\sigma k/2} \lesssim 2^{\sigma j/2},$ we have $J_{2,1} \lesssim\|f\|_{B_{\gamma^{\prime}, 2}^{\theta}L^{\rho^{\prime}}}$.  To estimate $J_{2,2}$, we interchange the order of the summation to obtain
	\begin{eqnarray*}
		J_{2,2}^{2} &=&\sum_{k=1}^{\infty} \sum_{j=1}^{k-1} 2^{(2 \theta-\sigma/2) (j-k)}\left\|\phi_{j} *_{t} \chi_{ k/4} *_{x} f\right\|_{L^{q_{1}, \infty}L^{r_{1}}}^{2} \\
		& \lesssim & \sum_{k=1}^{\infty}\left\|\chi_{ k/4} *_{x} f\right\|_{L^{q_1,\infty }L^{r_1}}^2 \lesssim\|f\|_{l^2L^{q_1,\infty }L^{r_1}}^2.
	\end{eqnarray*}
	Collecting these  estimates, we obtain
	$$
	J  \lesssim \|f\|_{B_{\gamma^{\prime}, 2}^{\theta}L^{\rho^{\prime}}}+\left\|f\right\|_{l^2L^{\overline{q}}L^{\overline{r}}}+\|f\|_{l^2L^{q_1,\infty }L^{r_1}}
	\lesssim |f\|_{B_{\gamma^{\prime}, 2}^{\theta}L^{\rho^{\prime}}}+\left\|f\right\|_{l^2L^{\overline{q}}L^{\overline{r}}},
	$$
	where we used Lemma \ref{f} when $\overline{r}\ge r$.
	This finishes the proof of (\ref{i3}).
	
	Finally,  the continuity of $Gf$ in time follows from the density argument. This completes the proof of Proposition \ref{p1}.
\end{proof}

\section{Proof of Theorem \ref{T1}}\label{s4}
In this section, we prove Theorem \ref{T1}. Firstly, we recall  two lemmas that we will need to complete the contraction argument.

\begin{lemma}[\cite{G}, Lemma 3.4]\label{l1}
	Assume $\alpha>0,0\le s<\alpha+1$, $f\in\mathcal{C}(\alpha)$.
	and  $1<p, r, \rho<\infty$ satisfy  $\fc1p=\fc\alpha \rho+\fc1r$. Then for any $u\in L^\rho\cap B^s_{r,2}$, we have\\
	(i) if $s\in\mathbb{Z}$, $$	\|f(u)\|_{ H^{s,p}}\lesssim\|u\|_{L^\rho}^\alpha\|u\|_{H^{s,r}},$$\\
	(ii) if $s\notin \mathbb{Z}$, $$	\|f(u)\|_{ B^s_{p,2}}\lesssim\|u\|_{L^\rho}^\alpha\|u\|_{B^s_{r,2}}.$$
\end{lemma}
\begin{lemma}[\cite{Na2}, Lemma 2.3]\label{l2}
	Assume $0<s<8,s\neq4$, $\max \left\{0,\frac{s}{4}-1\right\}<\alpha$ and $1\le\rho, r_0,r,q_0\le \infty $. Assume also that $1<\gamma, q<\infty $ and $\gamma ,\rho,q_0,r_0,q,r$ satisfy $\frac{1}{\gamma '}=\frac{\alpha}{q_0}+\frac{1}{q}, \frac{1}{\rho'}=\frac{\alpha}{r_0}+\frac{1}{r}$.
	Then for any $f\in \mathcal{C}(\alpha)$ and $u\in L^{q_0}L^{r_0}\cap B^{s/4}_{q,2}L^r$, we have
	\begin{equation}
		\left\|f(u)\right\|_{B^{s/4}_{\gamma ',2}L^{\rho'}}\lesssim \left\|u\right\|_{L^{q_0}L^{r_0}}^\alpha \left\|u\right\|_{B^{s/4}_{q,2}L^r}.\notag
	\end{equation}
\end{lemma}
We regard the solution of the Cauchy problem (\ref{NLS}) as the fixed point of the integral equation given by
\begin{equation}\label{Su}
	u(t)=(Su)(t)=e^{it(\Delta^2+\mu\Delta)}\phi+i \int_{0}^{t}e^{i\left(t-s\right)(\Delta^2+\mu\Delta)}f(u)\left(s\right)\mathrm{d}s,
\end{equation}
for $t\in\R$, where $u(t):=u(t,\cdot)$. Note that $Su$ satisfies
\begin{equation}\label{SNLS}
	\begin{cases}
		i\partial_t\left(Su\right)+\Delta^2\left(Su\right)+\mu\Delta \left(Su\right)+f(u)=0,\\
		\left(Su\right)(0)=\phi,
	\end{cases}
\end{equation}
and that
\begin{equation}\label{tNLS}
	\partial_t \left(Su\right)  = i e^{it(\Delta^2+\mu\Delta)}\left[\left(\Delta^2+\mu\Delta\right)\phi+f(\phi)\right]
	+i \int_{0}^{t}e^{i\left(t-s\right)(\Delta^2+\mu\Delta)}\partial_sf(u)\left(s\right)\mathrm{d}s.
\end{equation}

We now resume the proof of Theorem \ref{T1}. We consider  four cases: $0<s<4$, $s=4$, $4<s<6$ and $6\le s<8$.
\subsection{The case $0<s<4$}
Throughout this subsection, we fix
\begin{equation}\label{ga1}
	\gamma =\frac{2N+8}{N},\qquad \rho=\frac{2N+8}{N}.
\end{equation}
We then define $q,r,\overline{q},\overline{r}$ such that
\begin{equation}\label{1153}
	\overline{q}=\gamma ',\qquad\frac{4}{\overline{q}}-N\left(\frac{1}{2}-\frac{1}{\overline{r}}\right)=4-s
\end{equation}
and
\begin{equation}
	\frac{1}{\gamma '}=\frac{\alpha+1}{q},\qquad \frac{1}{\rho'}=\alpha \left(\frac{1}{r}-\frac{s}{N}\right)+\frac{1}{r}.\notag
\end{equation}
Since $\alpha=\frac{8}{N-2s}$, $0<s<\min \left\{\frac{N}{2},4\right\}$,  it is straightforward to verify that $(\gamma ,\rho), (q,r)\in\Lambda_b$ are two biharmonic admissible pairs, $1<\overline{q}<2<\overline{r}<\infty $, and $r<\frac{N}{s}$.

Assume $\left\|\phi\right\|_{H^s}$ sufficiently small such that
\begin{equation}\label{251}
	\left(2C_1\right)^{\alpha+1}\left\|\phi\right\|_{H^s}^\alpha\le1,\qquad   C_2\left(2C_1\left\|\phi\right\|_{H^s}\right)^\alpha\le \frac{1}{2},
\end{equation}
where $C_1,C_2$ are the constants in (\ref{148}) and (\ref{147}), respectively. Set $M=2C_1\left\|\phi\right\|_{H^s}$ and consider the metric space
$$
X_M=\left\{u\in L^\infty H^s\cap L^qB^s_{r,2}\cap B^{s/4}_{q,2}L^r:\left\|u\right\|_{L^\infty H^s \cap L^qB^s_{r,2}\cap B^{s/4}_{q,2}L^r}\le M\right\}.\notag
$$
It follows that $X_M$ is a complete metric space when equipped with the distance
\begin{equation}\label{d}
	d(u,v)=\left\|u-v\right\|_{L^\infty L^2\cap L^qL^r}.
\end{equation}
Next, we show that the map $S$, defined in (\ref{Su}), is a contraction on the space $X_M$.

We first show that $S$ maps $X_M$ into itself. Using  Proposition \ref{p1}, we get
\begin{equation}\label{144}
	\left\|Su\right\|_{L^\infty H^s\cap L^qB^s_{r,2}\cap B^{s/4}_{q,2}L^r}\lesssim \left\|\phi\right\|_{H^s}+\left\|f(u)\right\|_{B^{s/4}_{\gamma ',2}L^{\rho'}}+\left\|f(u)\right\|_{l^2L^{\overline{q}}L^{\overline{r}}}.
\end{equation}
Since $\frac{1}{\gamma '}=\frac{\alpha+1}{q}$, and $\frac{1}{\rho'}=\alpha \frac{N-sr}{Nr}+\frac{1}{r}$, we deduce from  Lemma \ref{l2}  and Sobolev's embedding $B^s_{r,2}(\R^N)\hookrightarrow L^{\frac{Nr}{N-sr}}(\R^N)$ that
\begin{equation}\label{145}
	\left\|f(u)\right\|_{B^{s/4}_{\gamma ',2}L^{\rho'}}\lesssim \left\|u\right\|_{L^q L^{\frac{Nr}{N-sr}}  }^\alpha \left\|u\right\|_{B^{s/4}_{q,2},L^r} \lesssim \left\|u\right\|_{L^qB^s_{r,2}}^\alpha \left\|u\right\|_{B^{s/4}_{q,2},L^r}.
\end{equation}
Next, we  estimate $\left\|f(u)\right\|_{l^2L^{\overline{q}}L^{\overline{r}}}$.
Since $r,\overline{r}>2$, we can choose $\ep>0$ sufficiently small and $\rho_\ep,q_\ep>2$ such that
\begin{equation}\label{1151}
	\frac{1}{\overline{r}}=\frac{1}{\rho_\ep}-\frac{\ep}{N},\qquad \frac{1}{q_\ep}=\frac{1}{r}-\frac{s-\ep}{N}.
\end{equation}
Then we deduce  from  Minkowski's inequality ($\overline{q}=\gamma '\le2$) and  Sobolev's embedding $B^\ep_{\rho_\ep,2}(\R^N)\hookrightarrow B^0_{\overline{r},2}(\R^N)$ that
\begin{equation}\label{141}
	\left\|f(u)\right\|_{l^2L^{\overline{q}}L^{\overline{r}}} \lesssim  \left\|f(u)\right\|_{L^{\gamma '}B^0_{\overline{r},2}}\lesssim \left\|f(u)\right\|_{L^{\gamma '}B^\ep_{\rho_\ep,2}}.
\end{equation}
Moreover, since $\frac{1}{\rho_\ep}=\alpha \left(\frac{1}{r}-\frac{s}{N}\right)+\frac{1}{q_\ep}$  by (\ref{ga1}), (\ref{1153}) and (\ref{1151}), it follows from Lemma \ref{l1}  and  Sobolev's embedding $B^s_{r,2}\left(\R^N\right)\hookrightarrow L^{\frac{Nr}{N-sr}}\left(\R^N\right)$ that
\begin{equation}\label{142}
	\left\|f(u)\right\|_{B^\ep_{\rho_\ep,2}}\lesssim \left\|u\right\|_{L^{\frac{Nr}{N-sr}}}^\alpha \left\|u\right\|_{B^\ep_{q_\ep,2}}\lesssim \left\|u\right\|_{B^s_{r,2}}^\alpha \left\|u\right\|_{B^\ep_{q_\ep,2}}.
\end{equation}
Estimates  (\ref{141}), (\ref{142}) and H\"older's inequality imply
\begin{equation}\label{146}
	\left\|f(u)\right\|_{l^2L^{\overline{q}}L^{\overline{r}}} \lesssim \left\|u\right\|_{L^qB^s_{r,2}}^\alpha \left\|u\right\|_{L^qB^\ep_{q_\ep,2}}\lesssim \left\|u\right\|_{L^qB^s_{r,2}}^{\alpha+1},
\end{equation}
where we used the embedding $B^s_{r,2}\left(\R^N\right)\hookrightarrow B^\ep_{q_\ep,2}\left(\R^N\right)$ (see (\ref{1151})) in the second inequality.
It now  follows from  (\ref{144}), (\ref{145}) and (\ref{146}) that, for any $u\in X_M$,
\begin{equation}\label{148}
	\left\|Su\right\|_{L^\infty H^s\cap L^qB^s_{r,2}\cap B^{s/4}_{q,2}L^r}\le C_1\left\|\phi\right\|_{H^s}+C_1\left\|u\right\|^{\alpha+1}_{L^qB^s_{r,2}\cap B^{s/4}_{q,2}L^r}\le M,
\end{equation}
where we used (\ref{251}) in the last inequality.

Our next aim is the desired Lipschitz property of $S$ with respect to the metric $d$ defined in (\ref{d}). For any $u,v\in X_M$, we deduce from Strichartz's estimate (\ref{SZ}), (\ref{fu}), (\ref{251}), H\"older's inequality and Sobolev's embedding that
\begin{eqnarray}\label{147}
	d(Su,Sv)&\lesssim &\left\|\left(\left|u\right|^{\alpha}+\left|v\right|^{\alpha}\right)\left(u-v\right)\right\|_{L^{\gamma '}L^{\rho'}}\notag\\
	&\lesssim &\left(\left\|u\right\|^\alpha_{L^qB^s_{r,2}}+\left\|v\right\|^\alpha_{L^qB^s_{r,2}}\right)\left\|u-v\right\|_{L^qL^r}\notag\\
	&\le& C_2M^\alpha d(u,v)\le\frac{1}{2}d(u,v).
\end{eqnarray}

Therefore,  by Banach's fixed point theorem, we conclude that the Cauchy problem (\ref{NLS}) admits a  unique global solution $u\in CH^s\cap L^qB^s_{r,2}\cap B^{s/4}_{q,2}L^r$, where the continuity of $u$ in time follows from Proposition \ref{p1}.
\subsection{The case $s=4$}
Throughout this subsection, we fix
\begin{equation}
	\gamma=\frac{8(\alpha+2)}{(N-8)\alpha},\qquad \rho=\frac{N(\alpha+2)}{N+4\alpha}.\notag
\end{equation} For $\phi\in H^4$ and $T>0$, we define
\begin{eqnarray}
	F(\phi,T)&=&\|e^{it(\Delta^2+\mu\Delta)}(\Delta^2+\mu\Delta)\phi\|_{L^{\gamma}\left([0,T],L^\rho\right)}+\|e^{it(\Delta^2+\mu\Delta)}f(\phi)\|_{L^{\gamma}\left([0,T],L^\rho\right)}\notag\\
	&&+\left\|e^{it(\Delta^2+\mu\Delta)}\phi\right\|_{L^\gamma \left([0,T], H^{4,\rho}\right)}. \notag
\end{eqnarray}
By Strichartz's estimate (\ref{sz}), (\ref{fu}) and Sobolev's embedding $H^4\left(\R^N\right) \hookrightarrow L^{2\alpha+2}\left(\R^N\right)$, we have
\begin{equation}\label{1261}
	F(\phi,T) \lesssim  \left\|\phi\right\|_{H^4}+\left\|f(\phi)\right\|_{L^2}
	\le  C_3 \left(\left\|\phi\right\|_{H^4}+\left\|\phi\right\|_{H^4}^{\alpha+1}\right).
\end{equation}
We then  recall the following result from  \cite{xuan}.
\begin{proposition}[Proposition 5.1 in \cite{xuan}]\label{p2}
	Let $N>8$, $\alpha=\frac{8}{N-8},\mu=0$ or $-1$ and $f\in\mathcal{C}(\alpha)$. There exists $M>0,C_4>0$ such that for any  $T>0$ with
	\begin{equation}\label{1281}
		C_4\left(1+\left\|\phi\right\|_{H^4}^\alpha\right)F(\phi,T)\le \frac{M}{2},
	\end{equation}
	the Cauchy problem (\ref{NLS}) admits a unique solution $u\in   C \left([0,T],H^4\right)\cap L^\gamma  \left([0,T],H^{4,\rho}\right)$ satisfying $\left\|u\right\|_{H^{1,\gamma }\left([0,T],L^\rho\right)\cap L^\gamma \left([0,T],H^{4,\rho}\right)}\le M$.
\end{proposition}

Let $\left\|\phi\right\|_{H^4}$ sufficiently small such that
\begin{equation}\label{1282}
	C_3C_4\left(1+\left\|\phi\right\|_{H^4}^\alpha\right)\left(\left\|\phi\right\|_{H^4}+\left\|\phi\right\|_{H^4}^{\alpha+1}\right)\le \frac{M}{2},
\end{equation}
where $C_3,C_4$ are the constants in (\ref{1261}) and (\ref{1281}) respectively.
It now follows from   Proposition \ref{p2}, (\ref{1261}) and (\ref{1282}) that  for any $T>0$, the Cauchy problem (\ref{NLS}) admits a unique global solution $u\in C \left([0,T],H^4\right)\cap L^\gamma  \left([0,T],H^{4,\rho}\right)$ with $\left\|u\right\|_{H^{1,\gamma }\left([0,T],L^\rho\right)\cap L^\gamma \left([0,T],H^{4,\rho}\right) }\le M$. Since $T>0$ is arbitrary and $M>0$ is fixed, we deduce that    (\ref{NLS}) admits a unique solution $u\in C \left([0,\infty ),H^4\right)\cap L^\gamma  \left([0,\infty ),H^{4,\rho}\right)$. By  symmetry, a similar conclusion is reached in the negative time direction. Therefore, we obtain a unique solution $u\in C H^4\cap L^\gamma  H^{4,\rho}$ to (\ref{NLS}).
\subsection{The case $4<s<6$}
Throughout this subsection, we fix
\begin{equation}
	\gamma =2, \qquad \rho=\frac{2N}{N-4}.\notag
\end{equation}
We then define $q,r,\overline{q},\overline{r}$ such that
\begin{equation}
	\overline{q}=2,\qquad\frac{4}{\overline{q}}-N\left(\frac{1}{2}-\frac{1}{\overline{r}}\right)=8-s,\notag
\end{equation}
and
\begin{equation}
	\frac{1}{\gamma '}=\frac{\alpha+1}{q},\qquad \frac{1}{\rho'}=\alpha \left(\frac{1}{r}-\frac{s}{N}\right)+\frac{1}{r}.\notag
\end{equation}
Since $4<s<6,N>2s$, it is straightforward to verify that $(\gamma ,\rho), (q,r)\in\Lambda_b$ are two biharmonic admissible pairs, $1<\overline{r}<2 $,   $r<\frac{N}{s}$ and
$\frac{1}{\overline{r}}=\alpha \left(\frac{1}{r}-\frac{s}{N}\right)+\frac{1}{r}-\frac{s-4}{N}$.

Assume $\left\|\phi\right\|_{H^s}$ sufficiently small such that
\begin{equation}\label{252}
	\left(2C_5\right)^{\alpha+1}\left(\left\|\phi\right\|_{H^s}+\left\|\phi\right\|_{H^s}^{\alpha+1}\right)^\alpha\le1,\qquad \left(C_6+C_7\right)\left(2C_5\left(\left\|\phi\right\|_{H^s}+\left\|\phi\right\|_{H^s}^{\alpha+1}\right)\right)^\alpha\le \frac{1}{2},
\end{equation}
where $C_5,C_6,C_7$ are the constants in (\ref{1420}), (\ref{1421}) and (\ref{1283}), respectively. Set $M=2C_5(\left\|\phi\right\|_{H^s}$ $+\left\|\phi\right\|_{H^s}^{\alpha+1})$ and consider the metric space
\begin{equation}\label{7273}\begin{array}{c}
		Y_{M}=\left\{u \in L^\infty H^s\cap L^qB^s_{r,2}\cap B^{s/4}_{q,2}L^r\cap H^{1,q}B^{s-4}_{r,2}:\right. \\
		\qquad\qquad\qquad\left.
		\left\|u\right\|_{L^\infty H^s\cap L^qB^s_{r,2}\cap B^{s/4}_{q,2}L^r\cap H^{1,q}B^{s-4}_{r,2}}\le M\right\}.
	\end{array}\notag
\end{equation}
It follows that $Y_M$ is a complete metric space when equipped with the distance
\begin{equation}\label{d2}
	d(u,v)=\left\|u-v\right\|_{L^\infty L^2\cap L^qL^r}.
\end{equation}
Next, we show that the map $S$, defined in (\ref{Su}), is a contraction on the space $Y_M$.

We first show that $S$ maps $Y_M$ into itself. From the equation (\ref{SNLS}), we have
\begin{equation}\label{149}
	\left\|Su\right\|_{L^\infty H^s}\le \left\|Su\right\|_{L^\infty L^2}+\left\|Su\right\|_{L^\infty H^{s-2}}+\left\|\partial_t\left(Su\right)\right\|_{L^\infty H^{s-4}}+\left\|f(u)\right\|_{L^\infty H^{s-4}}
\end{equation}
and
\begin{equation}\label{1410}
	\left\|Su\right\|_{L^qB^s_{r,2}}\le \left\|Su\right\|_{L^qL^r}+\left\|Su\right\|_{L^qB^{s-2}_{r,2}}+\left\|\partial_t\left(Su\right)\right\|_{L^qB^{s-4}_{r,2}}+\left\|f(u)\right\|_{L^qB^{s-4}_{r,2}}.
\end{equation}
Since $\left(H^s,L^2\right)_{2/s,2}=H^{s-2}$ and $\left(B^s_{r,2},B^0_{r,\infty }\right)_{2/s,2}=B^{s-2}_{r,2}$ (see Theorem 6.4.5 in \cite{Bergh}),  it follows from H\"older's inequality and Young's inequality that
\begin{equation}\label{1411}
	\left\|Su\right\|_{L^\infty H^{s-2}}\lesssim \left\|Su\right\|_{L^\infty H^s}^{1-2/s}\left\|Su\right\|_{L^\infty L^2}^{2/s}\le \frac{1}{2}\left\|Su\right\|_{L^\infty H^s}+C\left\|Su\right\|_{L^\infty L^2},
\end{equation}
and
\begin{equation}\label{1412}
	\left\|Su\right\|_{L^qB^{s-2}_{r,2}}\lesssim \left\|Su\right\|_{L^qB^s_{r,2}}^{1-2/s}\left\|Su\right\|_{L^qB^0_{r,\infty }}^{2/s}\le \frac{1}{2}\left\|Su\right\|_{L^qB^s_{r,2}}+C\left\|Su\right\|_{L^qL^r},
\end{equation}
where we used the embedding $L^r\left(\R^N\right) \hookrightarrow B^0_{r,\infty }\left(\R^N\right)$ (see Theorem 6.4.4 in \cite{Bergh}) in (\ref{1412}).
Estimates (\ref{149})--(\ref{1412}) imply
\begin{equation}\label{1415}
	\left\|Su\right\|_{L^\infty H^s}\leq \left\|Su\right\|_{L^\infty L^2}+\left\|\partial_t\left(Su\right)\right\|_{L^\infty H^{s-4}}+\left\|f(u)\right\|_{L^\infty H^{s-4}}
\end{equation}
and
\begin{equation}\label{1416}
	\left\|Su\right\|_{L^qB^s_{r,2}}\le \left\|Su\right\|_{L^qL^r}+\left\|\partial_t\left(Su\right)\right\|_{L^qB^{s-4}_{r,2}}+\left\|f(u)\right\|_{L^qB^{s-4}_{r,2}}.
\end{equation}
From (\ref{1415}) and (\ref{1416}), we have
\begin{eqnarray}\label{271}
	&&\left\|Su\right\|_{L^\infty H^s\cap L^qB^{s}_{r,2}\cap B^{s/4}_{q,2}L^r}\\
	&\lesssim& \left\|Su\right\|_{L^\infty L^2\cap L^qL^r}+\left\|f(u)\right\|_{L^\infty H^{s-4}\cap L^qB^{s-4}_{r,2}}
	+\left\|\partial_t  \left(Su\right)\right\|_{L^\infty H^{s-4}\cap L^qB^{s-4}_{r,2}\cap B^{\left(s-4\right)/4}_{q,2}L^r}.\notag
\end{eqnarray}
We first estimate $\left\|Su\right\|_{L^\infty L^2\cap L^qL^r}$. Since $\frac{1}{\gamma '}=\frac{\alpha+1}{q}, \frac{1}{\rho'}=\alpha \left(\frac{1}{r}-\frac{s}{N}\right)+\frac{1}{r}$, it follows from  Strichartz's estimates  (\ref{sz})--(\ref{SZ}), (\ref{fu}),  H\"older's inequality and Sobolev's embedding that
\begin{equation}\label{1417}
	\left\|Su\right\|_{L^\infty L^2\cap L^qL^r} \lesssim \left\|\phi\right\|_{L^2}+\left\|\left|u\right|^{\alpha}u\right\|_{L^{\gamma '}L^{\rho'}}
	\lesssim  \left\|\phi\right\|_{H^s}+\left\|u\right\|_{L^qB^s_{r,2}}^\alpha \left\|u\right\|_{L^qL^r}.
\end{equation}
Next, we estimate $\left\|f(u)\right\|_{L^\infty H^{s-4}\cap L^qB^{s-4}_{r,2}}$. Let
$p_1=\frac{2N}{N-8}$. Since $\frac{1}{2}=\alpha \frac{N-2s}{2N}+\frac{1}{p_1}$ and $\alpha+1>s-4$, we deduce from Lemma \ref{l1} and Sobolev's embedding $H^s \left(\R^N\right) \hookrightarrow B^{s-4}_{p_1,2}\left(\R^N\right)\cap L^{\frac{2N}{N-2s}}\left(\R^N\right)$ that
\begin{equation}\label{1413}
	\left\|f(u)\right\|_{L^\infty H^{s-4}}\lesssim \left\|\left\|u\right\|_{ \frac{2N}{N-2s}}^\alpha \left\|u\right\|_{B^{s-4}_{p_1,2}}\right\|_{L^\infty }\lesssim \left\|u\right\|_{L^\infty H^s}^{\alpha+1}.
\end{equation}
Let $p_2$ be given by   $\frac{1}{r}=\frac{\alpha(N-2s)}{2N}+\frac{1}{p_2}$. Similar to (\ref{1413}), we have
\begin{equation}\label{1414}
	\left\|f(u)\right\|_{L^qB^{s-4}_{r,2}}\lesssim \left\|u\right\|_{L^\infty \frac{2N}{N-2s}}^\alpha \left\|u\right\|_{L^qB^{s-4}_{p_2,2}}\lesssim \left\|u\right\|_{L^\infty H^s}^\alpha \left\|u\right\|_{L^qB^s_{r,2}}.
\end{equation}
Finally,  we claim that
\begin{eqnarray}
	\label{272}
	&&\left\|\partial_t\left(Su\right)\right\|_{L^\infty H^{s-4}\cap L^qB^{s-4}_{r,2}\cap B^{(s-4)/4}_{q,2}L^r}\notag\\
	&\lesssim & \left\|\phi\right\|_{H^s}+\left\|\phi\right\|_{H^s}^{\alpha+1}+\left\|u\right\|_{L^qB^{s}_{r,2}}^\alpha \left(\left\|u\right\|_{B^{s/4}_{q,2}L^r}+\left\|\partial_t  u\right\|_{L^q B^{s-4}_{r,2}}\right).
\end{eqnarray}
In fact, form  the equation (\ref{tNLS}) and the inequality (\ref{i2}), we have
\begin{eqnarray}\label{1197}
	&& \left\|\partial_t\left(Su\right)\right\|_{L^\infty H^{s-4}\cap L^qB^{s-4}_{r,2}\cap B^{(s-4)/4}_{q,2}L^r}\notag\\
	&\lesssim& \left\|\phi\right\| _{H^{s}}+\left\|f(\phi)\right\|_{H^{s-4}}+\left\|\partial_tf(u)\right\|_{B^{s/4-1}_{\gamma ',2}L^{\rho'}}+\left\|\partial_tf(u)\right\|_{l^2L^{\overline{q}}L^{\overline{r}}}.
\end{eqnarray}
Similar to (\ref{1413}), we have
\begin{equation}\label{1198}
	\left\|f(\phi)\right\|_{H^{s-4}}\lesssim \left\|\phi\right\|_{H^s}^{\alpha+1}.
\end{equation}
Moreover, we deduce from  Lemma \ref{l2} and Sobolev's embedding $B^s_{r,2}\left(\R^N\right)\hookrightarrow L^{\frac{Nr}{N-sr}}\left(\R^N\right)$ that
\begin{equation}\label{1165}
	\left\|\partial_tf(u)\right\|_{B^{s/4-1}_{\gamma ',2}L^{\rho'}}\lesssim \left\|f(u)\right\|_{B^{s/4}_{\gamma ',2}L^{\rho'}} \lesssim \left\|u\right\|_{L^qB^s_{r,2}}^\alpha \left\|u\right\|_{B^{s/4}_{q,2},L^r}.
\end{equation}
On the other hand, since $\overline{q}=\gamma '=2$ and $1<\overline{r}\le2$, it follows from  Minkowski's inequality and the embedding   $L^{\overline{r}}\left(\R^N\right)\hookrightarrow B^0_{\overline{r},2}\left(\R^N\right)$ (see Theorem 6.4.4 in \cite{Bergh}) that
\begin{equation}\label{11910}
	\left\|\partial_tf(u)\right\|_{l^2L^{\overline{q}}L^{\overline{r}}} \lesssim  \left\|\partial_tf(u)\right\|_{L^{\overline{q}}B^0_{\overline{r},2}}\lesssim \left\|\partial_tf(u)\right\|_{L^{\gamma '}L^{\overline{r}}}.
\end{equation}
Since $\frac{1}{\overline{r}}=\alpha \left(\frac{1}{r}-\frac{s}{N}\right)+\frac{1}{r}-\frac{s-4}{N}$, it follows from   (\ref{fu}), H\"older's inequality and Sobolev's embedding $B^s_{r,2}\left(\R^N\right)\hookrightarrow L^{\frac{Nr}{N-sr}}\left(\R^N\right), B^{s-4}_{r,2}\left(\R^N\right)\hookrightarrow L^{\frac{Nr}{N-(s-4)r}}\left(\R^N\right)$ that
$$
\left\|\partial_tf(u)\right\|_{L^{\overline{r}}}\lesssim \left\|u\right\|_{L^{\frac{Nr}{N-sr}}}^\alpha \left\|\partial_tu\right\|_{L^{\frac{Nr}{N-(s-4)r}}}\lesssim  \left\|u\right\|_{B^s_{r,2}}^\alpha \left\|\partial_tu\right\|_{B^{s-4}_{r,2}}.
$$
This inequality together with (\ref{11910}), H\"older's inequality implies
\begin{equation}\label{1419}
	\left\|\partial_tf(u)\right\|_{l^2L^{\overline{q}}L^{\overline{r}}}
	\lesssim \left\|u\right\|^\alpha_{L^q B^s_{r,2}}\left\|\partial_tu\right\|_{L^q B^{s-4}_{r,2}}.
\end{equation}
The inequality (\ref{272}) is now an immediate consequence of (\ref{1197}), (\ref{1198}), (\ref{1165}) and (\ref{1419}).

Estimates  (\ref{271})--(\ref{272}) imply that, for any $u\in Y_M$,
\begin{equation}\label{1420}
	\left\|Su\right\|_{L^\infty H^s\cap L^qB^s_{r,2}\cap B^{s/4}_{q,2}L^r \cap H^{1,q}B^{s-4}_{r,2}}\le C_5\left(\left\|\phi\right\|_{H^s}+\left\|\phi\right\|_{H^s}^{\alpha+1}\right)+C_5M^{\alpha+1}\le M,
\end{equation}
where we used (\ref{252}) in the second inequality.

Our next aim is the desired Lipschitz property of $S$ with respect to the metric $d$ defined in (\ref{d2}). Similar to (\ref{147}), we have for any $u,v\in Y_M$,
\begin{eqnarray}\label{1421}
	d(Su,Sv)&\lesssim&  \left(\left\|u\right\|^\alpha_{L^qB^s_{r,2}}+\left\|v\right\|^\alpha_{L^qB^s_{r,2}}\right)\left\|u-v\right\|_{L^qL^r}\notag\\
	&\le& C_6M^\alpha d(u,v)\le\frac{1}{2}d(u,v).
\end{eqnarray}

Therefore, we deduce from Banach's fixed point argument that the Cauchy problem (\ref{NLS}) admits a  unique global solution $u\in L^\infty H^s\cap L^qB^s_{r,2}\cap B^{s/4}_{q,2}L^r\cap H^{1,q}B^{s-4}_{r,2}$.

It remains to  prove that $u\in C\left(\R, H^s\right)$.  Similar to (\ref{1415}), we have
\begin{eqnarray}\label{273}
	\left\|u(t_1)-u(t_2)\right\|_{H^s}&\lesssim& \left\|\partial_t u(t_1)-\partial_t u(t_2)\right\|_{H^{s-4}}+\left\|u(t_1)-u(t_2)\right\|_{L^2}\notag\\
	&&+\left\|f\left(u(t_1)\right)-f\left(u(t_2)\right)\right\|_{H^{s-4}}.
\end{eqnarray}
Since $\partial_tf\in B^{s/4-1}_{\gamma ',2}L^{\rho'}\cap l^2L^{\overline{q}}L^{\overline{r}}$ by (\ref{1165}) and  (\ref{11910}), we deduce from  (\ref{tNLS}) and Proposition \ref{p1} that $u\in C^1\left(\R,H^{s-4}\right)$,
so that by (\ref{273}) it suffices to prove  $f(u)\in C\left(\R,H^{s-4}\right)$.

To this end, we first show that $f(u)\in C\left(\R,B^{0}_{\rho_0,\infty }\right)$, where $\rho_0$ is given by $\frac{1}{\rho_0}=\frac{1}{2}-\frac{s-4}{N}$.
Indeed, using the same method as that used to derive (\ref{1416}), we obtain
\begin{eqnarray}\label{1284}
	\left\|u(t_1)-u(t_2)\right\|_{H^{4,\rho_0}}&\lesssim& \left\|\partial_t u(t_1)-\partial_t u(t_2)\right\|_{L^{\rho_0}}+\left\|u(t_1)-u(t_2)\right\|_{L^{\rho_0}}\notag\\
	&&+\left\|f\left(u(t_1)\right)-f\left(u(t_2)\right)\right\|_{L^{\rho_0}}.
\end{eqnarray}
Moreover, it follows from  (\ref{fu}), H\"older's inequality, Sobolev's embedding $H^s\left(\R^N\right) \hookrightarrow H^{\frac{2N}{N-2s}}\left(\R^N\right)$ and $H^{4,\rho_0}\left(\R^N\right)\hookrightarrow L^{\frac{N\rho_0}{N-4\rho_0}}\left(\R^N\right)$ that
\begin{eqnarray}\label{1221}
	&&\left\|f\left(u(t_1)\right)-f\left(u(t_2)\right)\right\|_{L^{\rho_0}}\notag\\
	&\lesssim&\left(\left\|u(t_1)\right\|_{L^{\frac{2N}{N-2s}}}^\alpha+\left\|u(t_2)\right\|^\alpha_{L^\frac{2N}{N-2s}}\right)\left\|u(t_1)-u(t_2)\right\|_{L^{\frac{N\rho_0}{N-4\rho_0}}}\notag\\
	&\lesssim & \left\|u\right\|_{L^\infty H^s}^\alpha \left\|u(t_1)-u(t_2)\right\|_{H^{4,\rho_0}}.
\end{eqnarray}
Combining (\ref{1284}) and  (\ref{1221}), we obtain
\begin{eqnarray}\label{1283}
	\left\|u(t_1)-u(t_2)\right\|_{H^{4,\rho_0}}&\le& C_7\left\|\partial_t u(t_1)-\partial_t u(t_2)\right\|_{L^{\rho_0}}+C_7\left\|u(t_1)-u(t_2)\right\|_{L^{\rho_0}}\notag\\
	&&+ C_7\left\|u\right\|_{L^\infty H^s}^\alpha \left\|u(t_1)-u(t_2)\right\|_{H^{4,\rho_0}}.
\end{eqnarray}
Since $C_7\left\|u\right\|_{L^\infty H^s}^\alpha\le C_7M^\alpha\le \frac{1}{2}$ in (\ref{252}), we have
\begin{equation}\label{1222}
	\left\|u(t_1)-u(t_2)\right\|_{H^{4,\rho_0}}\lesssim  \left\|\partial_t u(t_1)-\partial_t u(t_2)\right\|_{L^{\rho_0}}+\left\|u(t_1)-u(t_2)\right\|_{L^{\rho_0}}.
\end{equation}
On the other hand, since $u\in C^1\left(\R,H^{s-4}\right)$ and $H^{s-4}\left(\R^N\right)\hookrightarrow L^{\rho_0}\left(\R^N\right)$, we have $u\in C^1\left(\R,L^{\rho_0}\right)$.  This together with (\ref{1222}) implies $u\in C\left(\R,H^{4,\rho_0}\right)$.  So by  (\ref{1221}) and Sobolev's embedding $L^{\rho_0}\left(\R^N\right) \hookrightarrow B^0_{\rho_0,\infty }\left(\R^N\right)$, we have   $f(u)\in C\left(\R,B^0_{\rho_0,\infty }\right)$.

We proceed to show $f(u)\in C\left(\R,H^{s-4}\right)$.  Let $\frac{1}{\rho_\ep}=\frac{1}{2}+\frac{\ep}{N}$ and $p_\ep=\frac{2N}{N-8+2\ep}$, where $\ep>0$ sufficiently small such that $\alpha>s-5+\ep$. We then claim that  $f(u)$ is bounded in $B^{s-4+\ep}_{\rho_\ep,2}$. In fact this follows from Lemma \ref{l2} ( $\frac{1}{\rho_\ep}=\alpha \frac{N-2s}{2N}+\frac{1}{p_\ep}$) and Sobolev's embedding $H^s\left(\R^N\right) \hookrightarrow L^{\frac{2N}{N-2s}}\left(\R^N\right)\cap B^{s-4+\ep}_{p_\ep,2}\left(\R^N\right)$,
\begin{equation}\label{1285}
	\left\|f(u)\right\|_{B^{s-4+\ep}_{\rho_\ep,2}}\lesssim \left\|u\right\|_{L^{\frac{2N}{N-2s}}}^\alpha \left\|u\right\|_{B^{s-4+\ep}_{p_\ep,2}}\lesssim \left\|u\right\|_{H^s}^{\alpha+1}.
\end{equation}
Then by the interpolation theorem  (see Theorem 6.4.5 in \cite{Bergh}), we have
$$
\left(B^0_{\rho_0,\infty },B^{s-4+\ep}_{\rho_\ep,2}\right)_{\theta,2}=B^{s-4}_{2,2}=H^{s-4}, \qquad \theta=\frac{s-4}{s-4+\ep}.
$$
This together with  (\ref{1285})  and the fact $f(u)\in C\left(\R,B^0_{\rho_0,\infty }\right)$ implies $f(u)\in C\left(\R,H^{s-4}\right)$. Combing (\ref{273}) and $u\in C^1\left(\R,H^{s-4}\right)$, we can immediately get that  $u\in C\left(\R, H^s\right)$.

\subsection{The case $6\leq s<8$}
Throughout this subsection, we fix $r=\frac{2N}{N-4}$.
Assume $\left\|\phi\right\|_{H^s}$ sufficiently small such that
\begin{equation}\label{253}
	\left(2C_8\right)^{\alpha+1}\left(\left\|\phi\right\|_{H^s}+\left\|\phi\right\|_{H^s}^{\alpha+1}\right)^\alpha\le1,\qquad C_9\left(2C_8\left(\left\|\phi\right\|_{H^s}+\left\|\phi\right\|_{H^s}^{\alpha+1}\right)\right)^\alpha\le \frac{1}{2},
\end{equation}
where $C_8,C_9$ are the constants in (\ref{11431}) and (\ref{11432}), respectively. Set $M=2C_8\left(\left\|\phi\right\|_{H^s}\right.$ $ \left.+\left\|\phi\right\|_{H^s}^{\alpha+1}\right)$ and consider the metric space
\begin{equation}
	Z_{M}=\left\{u \in L^\infty H^s\cap  B^{s/4}_{2,2}L^r\cap B^{(s-2)/4}_{2,2}B^2_{r,2}:
	\left\|u\right\|_{L^\infty H^s\cap  B^{s/4}_{2,2}L^r\cap B^{(s-2)/4}_{2,2}B^{2}_{r,2}}\le M\right\}.\notag
\end{equation}
It follows that $Z_M$ is a complete metric space when equipped with the distance
\begin{equation}\label{d3}
	d(u,v)=\left\|u-v\right\|_{L^\infty L^2\cap L^2L^r}.
\end{equation}
Next, we show that the map $S$, defined in (\ref{Su}), is a contraction on the space $Z_M$.

We first estimate $\left\|Su\right\|_{L^\infty H^s\cap B^{s/4}_{2,2}L^r}$.  Note that  (\ref{1415}), (\ref{1413}) and (\ref{1198}) still hold in the  case $6\le s<8$, we have
\begin{equation}	\label{274}
	\left\|Su\right\|_{L^\infty H^s}
	\lesssim \left\|u\right\|_{L^\infty H^s}^{\alpha+1} +\left\|Su\right\|_{L^\infty L^2}+\left\|\partial_t  \left(Su\right)\right\|_{L^\infty H^{s-4}},
\end{equation}
so that
\begin{equation}	\label{275}
	\left\|Su\right\|_{L^\infty H^s\cap B^{s/4}_{2,2}L^r} \lesssim  \left\|u\right\|_{L^\infty H^s}^{\alpha+1}
	+\left\|Su\right\|_{L^\infty L^2\cap L^2L^r}
	+\left\|\partial_t  \left(Su\right)\right\|_{L^\infty H^{s-4}\cap B^{\left(s-4\right)/4}_{2,2}L^r}.
\end{equation}
From (\ref{SNLS}) and Strichartz's estimate (\ref{SZ}),  we have
\begin{equation}\label{1286}
	\left\|Su\right\|_{L^\infty L^2\cap L^2L^r} \lesssim  \left\|\phi\right\|_{L^2}+\left\|f(u)\right\|_{L^2L^{r'}}
	\lesssim \left\|\phi\right\|_{H^s}+\left\|u\right\|_{L^\infty H^s}^\alpha \left\|u\right\|_{L^2L^r},
\end{equation}
where we used (\ref{fu}), H\"older's inequality and Sobolev's embedding $H^s\left(\R^N\right) \hookrightarrow L^{\frac{2N}{N-2s}}\left(\R^N\right)$ in the second inequality.
Next, we claim that
\begin{eqnarray}\label{276}
	&&\left\|\partial_t  \left(Su\right)\right\|_{L^\infty H^{s-4}\cap B^{\left(s-4\right)/4}_{2,2}L^r}\notag\\
	&\lesssim & \left\|\phi\right\|_{H^s}+\left\|\phi\right\|_{H^s}^{\alpha+1}+\left\|u\right\|_{L^\infty H^S}^\alpha \left(\left\|u\right\|_{B^{s/4}_{2,2}L^r}+\left\|u\right\|_{B^{\left(s-2\right)/4}_{2,2}B^{2}_{r,2}}\right).
\end{eqnarray}
In fact, from the equation (\ref{tNLS}), the inequalities (\ref{i2}) and (\ref{1198}), we have
\begin{eqnarray}\label{11421}
	&&\left\|\partial_t  \left(Su\right)\right\|_{L^\infty H^{s-4}\cap B^{\left(s-4\right)/4}_{2,2}L^r}\notag\\
	&\lesssim &\left\|\phi\right\|_{H^{s}}+\left\|f(\phi)\right\|_{H^{s-4}}+\left\|\partial_tf(u)\right\|_{B^{(s-4)/4}_{2,2}L^{r'}}+\left\|\partial_tf(u)\right\|_{l^2L^{4/(8-s)}L^2}\notag\\
	&\lesssim &\left\|\phi\right\|_{H^s}+\left\|\phi\right\|_{H^s}^{\alpha+1}+\left\|f(u)\right\|_{B^{s/4}_{2,2}L^{r'}}+\left\|\partial_tf(u)\right\|_{l^2L^{4/(8-s)}L^2}.
\end{eqnarray}
From Lemma \ref{l2} and Sobolev's embedding  $H^s \left(\R^N\right)\hookrightarrow L^{\frac{2N}{N-2s}}\left(\R^N\right)$, we have
\begin{equation}\label{11424}
	\left\|f(u)\right\|_{B^{s/4}_{2,2}L^{r'}} \lesssim  \left\|u\right\| _{L^\infty L^{\frac{2N}{N-2s}}}^\alpha\left\|u\right\|_{B^{s/4}_{2,2}L^r}
	\lesssim  \left\|u\right\|_{L^\infty H^s}^\alpha\left\|u\right\|_{B^{s/4}_{2,2}L^r}.
\end{equation}
It remains to estimate $\left\|\partial_tf(u)\right\|_{l^2L^{4/(8-s)}L^2}$.
Note that $\frac{4}{8-s}\ge 2$, so that
$B^{(s-6)/4}_{2,2}L^2\hookrightarrow B^{(s-6)/4}_{2,4/(8-s)}L^2\hookrightarrow L^{4/(8-s)}L^2$, which implies
\begin{equation}\label{11425}
	\left\|\partial_tf(u)\right\|_{l^2L^{4/(8-s)}L^2} \lesssim \left\|\partial_tf(u)\right\|_{l^2B^{(s-6)/4}_{2,2}L^2}\lesssim \left\|f(u)\right\|_{B^{(s-2)/4}_{2,2}L^2}.
\end{equation}
Moreover, it follows from  Lemma \ref{l2} and Sobolev's embedding $H^s\left(\R^N\right)\hookrightarrow L^{\frac{2N}{N-2s}}\left(\R^N\right)$, $B^2_{r,2}\left(\R^N\right)\hookrightarrow L^{\frac{2N}{N-8}}$ that
\begin{equation}\label{11426}
	\left\|f(u)\right\|_{B^{(s-2)/4}_{2,2}L^2} \lesssim  \left\|u\right\|^\alpha_{L^\infty L^{\frac{2N}{N-2s}}}\left\|u\right\|_{B^{(s-2)/4}_{2,2}L^{\frac{2N}{N-8}}}
	\lesssim  \left\|u\right\|^\alpha_{L^\infty H^s}\left\|u\right\|_{B^{(s-2)/4}_{2,2}B^2_{r,2}}.
\end{equation}
The inequality (\ref{276}) is now an immediate consequence of (\ref{11421})--(\ref{11426}). \\
Estimates  (\ref{275}), (\ref{1286}) and (\ref{276}) imply that, for any $u\in Z_M$,
\begin{equation}\label{11428}
	\left\|Su\right\|_{L^\infty H^s\cap B^{s/4}_{2,2}L^r}\lesssim \left\|\phi\right\|_{H^s}+\left\|\phi\right\|_{H^s}^{\alpha+1}+M^{\alpha+1}.
\end{equation}

We now estimate $\left\|Su\right\|_{B^{(s-2)/4}_{2,2}B^2_{r,2}}$. When $s=6$, we deduce from the inequality (\ref{i2}) and the equation (\ref{tNLS}) that
\begin{eqnarray}\label{261}
	&&\left\|\partial_t  \left(Su\right)\right\|_{L^2 B^{2}_{r,2}}\notag\\
	&\lesssim& \left\|\phi\right\|_{H^6}+\left\|f(\phi)\right\|_{H^2}+\left\|\partial_t  f(u)\right\|_{B^{1/2}_{2,2}L^{r'}}+\left\|\partial_t  f(u)\right\|_{l^2L^2L^2}\notag\\
	&\lesssim &\left\|\phi\right\|_{H^6}+\left\|\phi\right\|_{H^6}^{\alpha+1}+\left\|f(u)\right\|_{B^{3/2}_{2,2}L^{r'}}+\left\|\partial_t  f(u)\right\|_{L^2L^2},
\end{eqnarray}
where we used (\ref{1198}) and the embedding $L^2L^2 \hookrightarrow l^2L^2L^2$ in the last inequality.
When $6<s<8$, we deduce from  the equation (\ref{tNLS}), the inequality (\ref{i3}) ($\sigma=2,\theta=(s-4)/4$) and (\ref{1198}) that
\begin{eqnarray}\label{11429}
	&&\left\|\partial_t\left(Su\right)\right\|_{B^{(s-4)/4-2/4}_{2,2}B^2_{r,2}}\notag\\
	&\lesssim &\left\|\phi\right\|_{H^{s}}+\left\|f(\phi)\right\|_{H^{s-4}}+\left\|\partial_tf(u)\right\|_{B^{(s-4)/4}_{2,2}L^{r'}}+\left\|\partial_tf(u)\right\|_{l^2L^{4/(8-s)}L^2}\notag\\
	&\lesssim & \left\|\phi\right\|_{H^s}+\left\|\phi\right\|_{H^s}^{\alpha+1}+\left\|f(u)\right\|_{B^{s/4}_{2,2}L^{r'}}+\left\|\partial_tf(u)\right\|_{l^2L^{4/(8-s)}L^2}.
\end{eqnarray}
Estimates (\ref{261}),  (\ref{11429}), (\ref{11424}), (\ref{11425}) and (\ref{11426}) imply that, for any $u\in Z_{M}$,
\begin{equation}	\label{262}
	\left\|\partial_t\left(Su\right)\right\|_{B^{(s-4)/4-2/4}_{2,2}B^2_{r,2}}\lesssim \left\|\phi\right\|_{H^s}+\left\|\phi\right\|_{H^s}^{\alpha+1}+M^{\alpha+1}, \qquad 6\le s<8.
\end{equation}
On the other hand, it follows from the inequality (\ref{i2}) and the embedding $L^2L^2 \hookrightarrow l^2L^2L^2$ that
\begin{eqnarray}
	\label{277}
	\left\|Su\right\|_{L^2B^{2}_{r,2}}&\lesssim& \left\|\phi\right\|_{H^2}+\left\|f(u)\right\|_{B^{1/2}_{2,2}L^{r'}}+\left\|f(u)\right\|_{L^2L^2}\notag\\
	&\lesssim & \left\|\phi\right\|_{H^2}+\left\|u\right\|_{L^\infty H^s}^{\alpha} \left(\left\|u\right\|_{B^{s/4}_{2,2}L^r}+\left\|u\right\|_{B^{\left(s-2\right)/4}_{2,2}B^{2}_{r,2}}\right),
\end{eqnarray}
where we used (\ref{11424}) and (\ref{11426}) in the second inequality.
It now follows from (\ref{11428}), (\ref{262}), (\ref{277}) and (\ref{253}) that, for any $u\in Z_M$,
\begin{equation}\label{11431}
	\left\|Su\right\|_{L^\infty H^s\cap B^{s/4}_{2,2}L^r\cap B^{(s-2)/4}_{2,2}B^2_{r,2}}\le C_8\left(\left\|\phi\right\|_{H^s}+\left\|\phi\right\|_{H^s}^{\alpha+1}\right)+C_8M^{\alpha+1}\le M.
\end{equation}

Our next aim is the desired Lipschitz property of $S$ with respect to the metric $d$ defined in (\ref{d2}). For any $u,v\in Z_M$, we deduce from Strichartz's estimate (\ref{SZ}), the inequality (\ref{fu}), (\ref{253}), H\"older's inequality and Sobolev's embedding $H^s\left(\R^N\right)\hookrightarrow L^{\frac{2N}{N-2s}}\left(\R^N\right)$ that
\begin{eqnarray}\label{11432}
	d(Su,Sv)&\lesssim &\left\|\left(\left|u\right|^{\alpha}+\left|v\right|^{\alpha}\right)\left(u-v\right)\right\|_{L^2L^{r'}}\notag\\
	&\lesssim &\left(\left\|u\right\|^\alpha_{L^\infty H^s}+\left\|v\right\|^\alpha_{L^\infty H^s}\right)\left\|u-v\right\|_{L^2L^r}\notag\\
	&\le &C_9M^\alpha d(u,v)\le\frac{1}{2}d(u,v).
\end{eqnarray}

Therefore, we deduce from Banach's fixed point argument that the Cauchy problem (\ref{NLS}) admits a  unique global solution $u\in C\left(\R,H^s\right)\cap B^{s/4}_{2,2}L^r\cap B^{(s-2)/4}_{2,2}B^2_{r,2}$, where the continuity of $u$ in time follows from the same argument used  in the case $4<s<6$.





\section*{Funding:}  {This work is partially supported by the National Natural Science Foundation of China 11771389,  11931010 and 11621101.}

\end{document}